\documentclass[12pt]{amsart}
\usepackage{latexsym, enumerate}
\usepackage{amssymb,amsmath,mathtools}

\usepackage{tikz}
\usetikzlibrary{shapes.geometric,arrows}

\tikzstyle{box} = [rectangle,text centered, draw=black]
\tikzstyle{arrow} = [thick, ->, >=stealth]

\usepackage{color}

\usepackage{amsthm}

\usepackage{todonotes}

\usepackage{multirow}

\mathtoolsset{showonlyrefs}

\newcommand{\rd}{{\rm d}}
\newcommand{\e}{{\rm e}}
\newcommand{\Ran}{\mathop{\rm Ran}}

\newcommand{\R}{{\mathbb R}}
\newcommand{\C}{{\mathbb C}}
\newcommand{\Z}{{\mathbb Z}}

\newcommand\eps{\epsilon}
\newcommand\beq{\begin{equation}}
\newcommand\eeq{\end{equation}}
\newcommand{\beqnt}{\begin{equation*}}
\newcommand{\eeqnt}{\end{equation*}}

\newcommand\im{\mathrm{Im}\,}
\newcommand\I{\mathrm{i}}

\DeclareMathOperator{\supp}{supp}

\DeclareMathOperator{\Div}{div}

\DeclareMathOperator{\pv}{p.v.}
\newcommand{\mc}[1]{\mathcal{#1}}

\newtheorem{theorem}{Theorem}[section]

\newtheorem{definition}[theorem]{Definition}
\newtheorem{proposition}[theorem]{Proposition}
\newtheorem{corollary}[theorem]{Corollary}
\newtheorem{lemma}[theorem]{Lemma}
\newtheorem{remark}[theorem]{Remark}

\begin{document}

\title[From spectral cluster to resolvent estimates]{From spectral cluster to uniform resolvent estimates on compact manifolds}

 \author{Jean-Claude Cuenin}
 \address{Department of Mathematical Sciences, Loughborough University, Loughborough,
 Leicestershire, LE11 3TU United Kingdom}
 \email{J.Cuenin@lboro.ac.uk}
 
 \date{November 12, 2020}

\begin{abstract}
It is well known that uniform resolvent estimates imply spectral cluster estimates. We show that the converse is also true in some cases. In particular, Sogge's universal spectral cluster estimates for the Laplace--Beltrami operator on closed Riemannian manifolds directly imply uniform resolvent estimates outside a parabolic region, without any reference to parametrices. The method is purely functional analytic and takes full advantage of the known spectral cluster bounds. This yields new resolvent estimates for manifolds with boundary or with low-regularity metrics, among other examples. Moreover, we show that the resolvent estimates are stable under perturbations and use this to establish uniform Sobolev and spectral cluster inequalities for Schrödinger operators with singular potentials.
\end{abstract}

\maketitle

\section{Introduction}  

Let $\Delta_g$ be the Laplace-Beltrami operator on a compact boundaryless Riemannian manifold $(M,g)$ of dimension $n\geq 3$. Dos Santos Ferreira, Kenig and Salo \cite{MR3200351} proved the uniform Sobolev inequality
\begin{align}\label{Dos Santos Ferreira, Kenig and Salo}
\|u\|_{L^{\frac{2n}{n-2}}(M)}\lesssim\|(\Delta_g+(\lambda+\I)^2)u\|_{L^{\frac{2n}{n+2}}(M)}.
\end{align}
This generalizes a celebrated result of Kenig, Ruiz and Sogge \cite{MR894584} in the Euclidean case to the setting of compact manifolds without boundary. For the torus the inequality \eqref{Dos Santos Ferreira, Kenig and Salo} is due to Shen \cite{MR2366961}. Bourgain, Shao, Sogge and Yao \cite{BourgainShaoSoggeEtAl2015} obtained improved bounds on negatively curved manifolds and on the torus. Using similar techniques, Krupchyk and Uhlmann \cite{MR3285241} generalized \eqref{Dos Santos Ferreira, Kenig and Salo} to higher order elliptic differential operators satisfying a suitable curvature condition. Inequality \eqref{Dos Santos Ferreira, Kenig and Salo} is equivalent to a $L^{q'}\to L^{q}$ estimate for the resolvent operator $(\Delta_g+(\lambda+\I)^2)^{-1}$, with $q=2n/(n-2)$. Frank and Schimmer \cite{MR3620715}, and, independently, Burq, Dos Santos Ferreira and Krupchyk \cite{MR3848231}, proved the endpoint version
\begin{align}\label{Frank and Schimmer}
\|(\Delta_g+(\lambda+\I)^2))^{-1}\|_{L^{q'}(M)\to L^{q}(M)}\lesssim \langle\lambda\rangle^{2\sigma(q)-1},
\end{align}
where $\langle\lambda\rangle=2+|\lambda|$, $q=q_n:=2(n+1)/(n-1)$ is the critical exponent, and 
\begin{align}\label{def. sigma(q)}
\sigma(q)=\max\bigl( \, n(\tfrac12-\tfrac1q)-\tfrac12, \, \tfrac{n-1}2(\tfrac12-\tfrac1q)\, \bigr).
\end{align}
In fact, all other estimates \eqref{Frank and Schimmer} for $2\leq q\leq 2n/(n-2)$ follow from the endpoint case by Sobolev embedding and by interpolation with the trivial bound for $q=2$. In particular, this yields \eqref{Dos Santos Ferreira, Kenig and Salo} since $2\sigma(q)-1=0$ for $q=2n/(n-2)$. 

It is also well known (see e.g.\ \cite[Lemma 10]{MR3620715}) that \eqref{Frank and Schimmer} implies Sogge's \cite{MR930395} spectral cluster estimates
\begin{align}\label{Sogge}
\|\Pi_{[\lambda,\lambda+1]}\|_{L^2(M)\to L^{q}(M)}\lesssim \langle\lambda\rangle^{\sigma(q)}
\end{align}
for $2\leq q\leq \infty$. Here $\Pi_{[\lambda,\lambda+1]}:=\mathbf{1}_{[\lambda,\lambda+1]}(\sqrt{-\Delta_g})$, defined by the functional calculus, is the orthogonal projection onto the subspace of functions in $L^2(M)$ that are spectrally localized to frequencies in the unit length window $ [\lambda,\lambda+1]$. The $q=\infty$ bound is related to the sharp counting function remainder estimate of Avakumovi\'c, Levitan and Hörmander \cite{MR80862,MR0058067,MR0609014}. On general manifolds with boundary, Grieser \cite{MR1924468} was the the first to prove $L^{\infty}$ estimates for eigenfunctions (but not spectral clusters) without convexity or concavity assumptions.


It does not seem to be known whether the converse implication $(3)\implies (2)$ (or even $(3)\implies (1)$) is true in full generality. To quote the authors of \cite{MR3200351}: 
\begin{center}
``\emph{It seems that such estimates cannot be obtained as a direct consequence of the spectral cluster estimates obtained by Sogge}". 
\end{center}
This forced them to redo the Hadamard parametrix construction of \cite{MR930395}. 
The paper \cite{MR3200351} provides a fairly complete study of the mapping properties of the Hadamard parametrix for the resolvent and the corresponding remainder term, using oscillatory integral theorems. Only a special case of these estimates is needed to prove \eqref{Dos Santos Ferreira, Kenig and Salo}. Frank and Schimmer~\cite{MR3620715} 
avoided the explicit construction of the Hadamard parametrix, but used some of its off-diagonal (non self-dual) estimates established in \cite{MR3200351}. The approach of \cite{BourgainShaoSoggeEtAl2015} avoids oscillatory integral theorems and only uses stationary phase to prove \eqref{Dos Santos Ferreira, Kenig and Salo}. Sogge's \cite{MR930395} universal spectral cluster bounds are used to handle a ``global term". However, the proof still uses the short-time Hadamard parametrix for the wave propagator $\cos(t\sqrt{-\Delta_g})$ to handle a ``local term". The bounds for the latter are independent of the global geometry of the manifold. One of our main results, Theorem \ref{theorem spectral cluster to resolvent intro} (or rather its generalization, Theorem \ref{theorem spectral cluster to resolvent stronger version}), is close to \cite[Theorem 1.3]{BourgainShaoSoggeEtAl2015}, which shows that improvements of the $L^{q'}\to L^q$ resolvent bounds for $(\Delta_g+(\lambda+\I\mu)^2)^{-1}$ to a larger region $|\mu|\geq \epsilon(\lambda)$, with $\epsilon(\lambda)\leq 1$ and $\lambda\geq 1$, are equivalent to $L^2\to L^q$ spectral cluster bounds in shrinking frequency windows $[\lambda,\lambda+\epsilon(\lambda)]$. 
The first difference between this and our result is that we are in a sense concerned with the opposite of improved bounds. Namely, we consider metrics of low regularity and seek bounds on the resolvent without losses compared to the smooth case. In fact, we will show how the two approaches can be combined to get improved resolvent bounds from improved spectral cluster bounds, without any a priori assumption on the smoothness of the metric. 
Of course, it is unreasonable to expect improved estimates for metrics of such low regularity as considered here, but the point is that all the geometric information is now encoded in the spectral cluster bound.
The second difference is that we completely avoid parametrix estimates and, again, solely rely on the spectral cluster bounds (or the equivalent quasimode bounds). In the case of improved estimates, we replace the step in \cite{BourgainShaoSoggeEtAl2015} involving the Hadamard parametrix by a more abstract argument. As an application, we convert the recent spectral cluster bounds of Canzani and Galkowski \cite{Canzani--Galkowski} to improved resolvent bounds under pure dynamical assumptions on the manifold.
The Hadamard parametrix construction breaks down for metrics with limited smoothness. 
Smith \cite{MR1644105,MR2262171} and Tataru \cite{MR1749052,MR1833146,MR1887639} used wave packet techniques to construct parametrices for regularized wave operators
and established Strichartz and square function estimates for the variable-coefficient wave equation with rough coefficients. In contrast to the smooth case, one cannot hope for asymptotic expansions of the fundamental solution modulo smoothing operators since the errors in the parametrix are too large. However, for the sake of proving estimates, they can be considered as forcing terms, and the equation can be solved iteratively. The wave packet parametrices require roughly two derivatives on the coefficients as this guarantees uniqueness of the bicharacteristic flow. For metrics with $C^{1,1}$ coefficients one could prove the parametrix bounds for the resolvent estimates concurrently with those for the spectral cluster bounds.
This is essentially the approach of Koch and Tataru \cite[Theorem 2.5]{MR2094851}. The difference between the $TT^*$ version of the spectral cluster bound (or quasimode bound) and the resolvent is merely a causal factor $\mathbf{1}_{t\geq 0}$, where $t$ denotes one of the spatial coordinates (see the proof of Lemma \ref{lemma first order 2}). The dispersive ($L^1\to L^{\infty}$) estimate \cite[Proposition 7.4]{MR2094851} is insensitive to this factor, and the spectral cluster and resolvent estimates both follow from this.
However, this strategy fails to prove resolvent analogues of the spectral cluster bounds of Smith and Sogge \cite{MR2316270} or Smith, Koch and Tataru \cite{MR2443996}. In those cases, the metric is only Lipschitz, and good parametrix bounds hold only on small spatial scales. Highly nontrivial energy propagation estimates are needed to prove estimates on the unit scale. Following the arguments of these papers carefully, one could perhaps prove resolvent estimates directly. However, the preceding discussion should make it clear that this is not completely straightforward.
Theorem \ref{theorem spectral cluster to resolvent intro} has the advantage that it treats the spectral cluster estimates as a black box and thus yields several new resolvent estimates in one stroke. 

A simple argument involving the Phragmén–Lindelöf maximum principle (see e.g.\ \cite{MR3848231}) shows that \eqref{Frank and Schimmer} is equivalent to the bound
\begin{align}\label{Phragmen}
\|(\Delta_g+z)^{-1}\|_{L^{q'}(M)\to L^{q}(M)}\lesssim_{\delta} |z|^{\sigma(q)-1/2}
\end{align}
for arbitrary $\delta>0$ and all $z\in\C\setminus[0,\infty)$ with $\im\sqrt{z}\geq \delta$. Here we take the branch of the square root with $\im\sqrt{\cdot}>0$. After possibly multiplying the metric by an innocuous factor, we may always assume that $\delta=1$, in accordance with \eqref{Frank and Schimmer}. Because $\sigma(q)$ is decreasing in $1/q$,  the exponent of $|z|$ is always nonpositive. Hence these estimates are uniform in the region\footnote{This corresponds to the outside of a parabola, see \cite{MR3200351,BourgainShaoSoggeEtAl2015}.} $\im\sqrt{z}\geq 1$, which is why they are called \emph{uniform} resolvent estimates. We will also call the resolvent version of \eqref{Dos Santos Ferreira, Kenig and Salo} a ``uniform Sobolev inequality". 
Spectral cluster estimates \eqref{Sogge} with frequency windows of unit length are called ``universal" since they hold on any compact boundaryless Riemannian manifold. These are best possible in terms of the exponent $\sigma(q)$, see \cite[Chapter 5]{MR3645429}. However, they are rarely optimal for individual eigenfunctions (see e.g.\ \cite{MR1924569,MR3987179}). If $M$ is the sphere or, more generally, a Zoll manifold, then no improvements are possible, due to the high multiplicity of eigenvalues. Similarly, \eqref{Phragmen} is best possible in terms of the lower bound $\im\sqrt{z}\geq 1$ for these manifolds \cite{BourgainShaoSoggeEtAl2015,MR3620715}. In the low-regularity setting that we will consider here, $g\in C^s$, there may be losses in the exponents of \eqref{Sogge} and \eqref{Phragmen}, i.e.\ $\sigma(q)$ may have to be replaced by $\gamma(q)>\sigma(q)$. We will refer to spectral cluster estimates with exponent as in the smooth case, i.e.\ $\gamma(q)=\sigma(q)$, as ``perfect" estimates.\footnote{This terminology is borrowed from a seminar talk of Daniel Tataru.}
For the resolvent, we distinguish two types of perfect estimates. The first we call estimates with ``perfect exponents". This means that the exponent is the same as in \eqref{Phragmen} (and hence the same as for elliptic estimates), but the region $\Omega\subset\C$ for which the estimate holds may be smaller than the ``perfect region" $\Omega=\{\im\sqrt{z}\geq 1\}$. In the second type the region $\Omega$ is perfect, but the exponent may be larger than in \eqref{Phragmen}. If both the exponent and the region is perfect, we simply say that we have perfect resolvent estimates. Note that, by Sobolev embedding, if we have perfect exponents for $q_1$, then the same is true for all $q_2\geq q_1$. We will consider regions of the form $\Omega_{\rho}=\{z=(\lambda+\I\mu)^2:\,\lambda\geq 1,\,|\mu|\gtrsim \lambda^{\rho}\}$, disregarding  the elliptic region (the outside of the union of a disk around the origin and a cone in the direction of the positive $z$-axis) in which \eqref{Phragmen} trivially holds. Hence the relevant range for $\rho$ is $0\leq \rho\leq 1$, and the perfect region corresponds to $\rho=0$.

When proving \eqref{Frank and Schimmer} in a coordinate patch of the manifold, the crucial region in phase space is $|x|\lesssim 1$, $|\xi|\approx \lambda$. Away from this region, one has stronger elliptic estimates, and the inequality follows by Sobolev embedding. For the arguments concerning elliptic estimates we refer to the proof of \cite[Corollary 5]{MR2252331} or that of \cite[Theorem 7]{MR2280790} but we will discuss these briefly in Subsection \ref{subs. Elliptic estimates}. In the crucial phase space region, the difference between spectral cluster and resolvent estimates is at most logarithmic in the frequency $\lambda$. In fact, this holds in a rather abstract setting, see Proposition \ref{proposition abstract bounds}. Similar abstract considerations (see e.g.\ \cite[Theorem XIII.25]{MR0493421} or \cite[Corollary 3.3]{MR3211800}) show that the spectral cluster bound \eqref{Sogge} is equivalent to the analogue of \eqref{Frank and Schimmer} for the imaginary part of the resolvent. Hence, any difference between \eqref{Frank and Schimmer} and \eqref{Sogge} can only come from the real part. The logarithmic loss incurred by the abstract bounds hints at an inefficiency in detecting cancellations. 
This is most clearly seen in the constant coefficient case, where the distributional formula
\begin{align*}
(|\xi|^2-(\lambda+\I 0)^2)^{-1}=\pv(|\xi|^2-\lambda^2)^{-1}+\I\pi\delta(|\xi|^2-\lambda^2)
\end{align*}
reveals cancellations of the real part. More generally, in the case where the operator has only absolutely continuous spectrum, Guillarmou and Hassell \cite[Lemma 4.1]{MR3211800} showed that the cancellation effects can be exploited by an integration by parts argument. This requires stronger pointwise bounds on the spectral measure and all its derivatives of order $\leq n/2$. The cancellation is related to the causal structure of the resolvent mentioned before. 

Most modern proofs of spectral cluster and resolvent estimates rely on wave equation techniques. 
One popular approach is via square function estimates for the wave equation \cite{MR1168960,MR2262171,MR2457396,MR3900030,Smith resolvent}, another one is via Strichartz estimates for a half wave equation \cite{MR2094851,MR3848231,MR2342881,MR3282983}. The latter is purely stationary, i.e.\ all the variables are spatial. Locally in phase space, one can single out a distinguished direction that plays the role of a fictional time. In fact, the papers listed under the first approach also use this phase space localization. 
While square function estimates are of independent interest, the second approach seems more natural for a purely stationary problem. More importantly, the homogeneous Strichartz estimates (corresponding to spectral cluster bounds) can easily be upgraded to inhomogeneous estimates (corresponding to resolvent bounds) by invoking the Christ--Kiselev lemma \cite{MR1809116,MR1789924,MR2233925}. For square function estimates, the passage from homogeneous to inhomogeneous estimates is more involved, as the recent work of Smith \cite{Smith resolvent} shows.

\section{Main results}

\subsection{From spectral cluster to resolvent bounds}

\begin{theorem}\label{theorem spectral cluster to resolvent intro}
Let $(M,g)$ be a compact Riemannian manifold without boundary and with Lipschitz metric $g$. Assume that $2\leq q\leq 2n/(n-2)$ if $n\geq 3$ and $2\leq q<\infty$ if $n=2$, and that the spectral cluster estimate
\begin{align}\label{abstract spectral cluster estimate into}
\|\Pi_{[\lambda,\lambda+1]}\|_{L^2(M)\to L^{q}(M)}\lesssim \langle\lambda\rangle^{\gamma(q)}
\end{align}
holds for some $\gamma(q)\in[\sigma(q),\sigma(q)+1/2]$. Then we have the uniform resolvent estimate
\begin{align}\label{abstract resolvent estimate into}
\|(\Delta_g+(\lambda+\I)^2)^{-1}\|_{L^{q'}(M)\to L^{q}(M)}\lesssim \langle\lambda\rangle^{2\gamma(q)-1}.
\end{align}
\end{theorem}

\begin{remark}
The restrictions on $\gamma(q)$ are natural since, 
by Sobolev embedding, the bounds trivially hold for $\gamma(q)=\sigma(q)+1/2$. On the other hand, we cannot have better bounds than in the smooth case.
\end{remark}

\begin{remark}
We actually prove a more general result (Theorem \ref{theorem spectral cluster to resolvent stronger version}) that allows us to 1) obtain resolvent estimates with perfect exponents in a an intermediate region $\Omega_{\rho}\subset\C$, even when perfect spectral cluster estimates are not available and 2) convert improved spectral cluster bounds to improved resolvent bounds (in the sense discussed in the introduction).
\end{remark}

\begin{remark}\label{remark quadratic form}
We do not touch here on operator-theoretic considerations related to the definition of $\Delta_g$ and domain issues. In the context of low-regularity metrics, it is advantageous to define the operator by quadratic form methods, see e.g.\ \cite[Chapter 6]{MR1349825}. Thus, we implicitly assume that $\Delta_g$ is in divergence form and that the resolvent equation $\Delta_g u+(\lambda+\I)^2u=F$ is interpreted in a weak sense (see Subsection \ref{subsec. Localization and reduction to an equation on Rn} for more details). This point of view is also taken in \cite{MR2280790, MR2443996, MR3282983}.
\end{remark}

\subsection{Stability under perturbations}

\begin{theorem}\label{theorem stability under perturbations non-effective intro}
Assume that the conditions of Theorem \ref{theorem spectral cluster to resolvent intro} are satisfied with perfect exponent $\gamma(q)=\sigma(q)$ and that $V\in L^p(M)$ with $1/p=1/q'-1/q$. Then there exists $\Lambda=\Lambda(M,q,V)\geq 1$ such that
\begin{align}\label{Perfect resolvent estimates with V}
\|(\Delta_g+V+(\lambda+\I)^2)^{-1}\|_{L^{q'}\to L^q}\lesssim \langle\lambda\rangle^{2\sigma(q)-1}
\end{align}
for $\lambda\geq \Lambda$. Moreover, if $n=2$ or $n\geq 3$ and $q<2n/(n-2)$, then $\Lambda$ depends on $V$ only though $\|V\|_{L^p}$. 
\end{theorem} 

\begin{remark}
The above also follows from recent results by Blair, Huang, Sire and Sogge \cite{blair2020uniform}. However, for the exponents $2\leq q\leq 2n/(n-2)$ and the self-dual estimates considered here, we can relax the assumptions in the abstract Theorem 2.1 of \cite{blair2020uniform}. Instead of assuming the quasimode \emph{and} the uniform Sobolev estimate for the unperturbed operator, (2.2) and (2.6) in \cite{blair2020uniform}, we only need to assume one of them (since they are equivalent).
\end{remark}

\begin{remark}
Since $M$ is compact, Hölder's inequalities implies that $L^{p_2}(M)\subset L^{p_1}(M)$ for $p_2\geq p_1$. We will subsequently consider critically singular potentials $V\in L^{n/2}(M)$. For fixed $V$, the restriction $\lambda\geq \Lambda$ can easily be removed since the inequality for $\lambda\leq \Lambda$ is trivial.
\end{remark}

\begin{remark}
The stability holds in greater generality, not just for potential perturbations, see Proposition \ref{proposition stability under perturbations effective}. Similar results may be found in \cite{MR2094851,MR2252331}.
\end{remark}

\begin{corollary}
Perfect spectral cluster estimates (i.e.\ \eqref{abstract spectral cluster estimate into} with $\gamma(q)=\sigma(q)$) are stable under perturbations by real-valued potentials $V\in L^{n/2}(M)$.
\end{corollary}

\subsection{Manifolds with nonsmooth metrics}

\begin{theorem}\label{theorem nonsmooth into}
Let $\Delta_g$ be the Laplace--Beltrami operator on an $n$-dimensional compact boundaryless Riemannian manifold with $C^{s}$ metric\footnote{Lipschitz if $s=1$ and $C^{1,1}$ if $s=2$.}, $0\leq s\leq 2$. Assume that $2\leq q\leq 2n/(n-2)$ if $n\geq 3$ and $2\leq q<\infty$ if $n=2$. 
\begin{itemize}
\item[(i)] If $s=2$, then we have the perfect estimates
\begin{align}\label{Lq'Lq C11 intro}
\|(\Delta_g+(\lambda+\I)^2)^{-1}\|_{L^{q'}\to L^q}\lesssim \langle\lambda\rangle^{2\sigma(q)-1}.
\end{align}
\item[(ii)] If $s=1$ and $n=2$, then \eqref{Lq'Lq C11 intro} holds for $q>8$, and
\begin{align}\label{Lq'Lq Lipschitz intro q=8}
\|(\Delta_g+(\lambda+\I)^2)^{-1}\|_{L^{\frac{8}{7}}\to L^8}\lesssim \langle\lambda\rangle^{2\sigma(q)-1}\ln^3\langle\lambda\rangle.
\end{align}
\item[(iii)] For any $0\leq s\leq 2$ and with $\rho=\frac{2-s}{2+s}$, we have 
\begin{align}\label{Lq'Lq Cs intro}
\|(\Delta_g+(\lambda+\I\lambda^{\rho})^2)^{-1}\|_{L^{q'}\to L^q}\lesssim \langle\lambda\rangle^{2\sigma(q)-1}.
\end{align}
\end{itemize}
These estimates are stable under perturbations by real-valued potentials $V\in L^{n/2}(M)$.
\end{theorem}
\begin{remark}
$(i)$ follows from the perfect spectral cluster estimates for $C^{1,1}$ metrics, due to Smith \cite{MR2262171}. The resolvent estimates in this case are implicit in the work of Koch and Tataru \cite[Theorem 2.5]{MR2094851}.
\end{remark}

\begin{remark}
Counterexamples of Smith and Sogge \cite{MR1306017} show that, in the Hölder scale, $C^{1,1}$ is best possible for perfect spectral cluster estimates in the range $2\leq q\leq q_n$. This also implies that this regularity is optimal for the endpoint $L^{q_n'}\to L^{q_n}$ resolvent estimate. It is an interesting question what the minimal regularity for the uniform Sobolev inequality \eqref{Dos Santos Ferreira, Kenig and Salo} is.
\end{remark}

\begin{remark}
Recently, Smith \cite{Smith resolvent}, extending results of Chen and Smith \cite{MR3900030}, obtained uniform $L^p\to L^q$ resolvent estimates for certain (not necessarily dual) exponents obeying the Sobolev gap condition $1/p-1/q=2/n$ (thus, including \eqref{Dos Santos Ferreira, Kenig and Salo}, but not \eqref{Frank and Schimmer}) for manifolds of bounded sectional curvature. This setting is more general than that of $C^{1,1}$ metrics since the latter give rise to bounded curvature. We should emphasize here that our proof only works for self-dual estimates. Using the spectral cluster bounds of \cite{MR3900030}, we could also obtain the results of $(i)$ under the bounded curvature condition.
\end{remark}

\begin{remark}
For Lipschitz metrics, Koch, Smith and Tataru \cite{MR3282983} proved sharp spectral cluster estimates in two dimensions. Perfect estimates hold for $q> 8$ (with a logarithmic loss at the endpoint $q=8$). This accounts for the resolvent estimate in $(ii)$.
The higher dimensional bounds of~\cite{MR3282983} do not yield uniform resolvent bounds. The relevant issue is whether one has perfect estimates for the Sobolev exponent $q=2n/(n-2)$. The results in~\cite{MR3282983} only yield perfect estimates for larger exponents $q>(6n-2)/(n-1)$. It is conjectured that perfect estimates should hold for $q\geq 2(n+2)/(n-1)$. If true, this would imply the uniform Sobolev inequality \eqref{Dos Santos Ferreira, Kenig and Salo} in dimensions $n=3,4$. 
\end{remark}

\begin{remark}
All our estimates have perfect exponents, and in $(i),(ii)$ we also have perfect regions. In $(iii)$, we obtain resolvent estimates in an intermediate (between the perfect and the trivial) region $\Omega_{\rho}\subset\C$. This is sharp at the critical exponent $q=q_n$ for all $s\in [0,2]$ since the spectral cluster estimates in this case are sharp \cite{MR2280790}. It is also sharp in the trivial case $s=0$ as the spectral cluster bounds are no better than what can be obtained by Sobolev embedding \cite{MR1037600}. In all other cases, $(iii)$ is presumably not sharp. The proof uses results of Smith \cite{MR2280790} for $1\leq s\leq 2$ and of Koch, Smith and Tataru \cite{MR2289621} for $0\leq s\leq 1$. A minor obstruction in the latter case is that Theorem \ref{theorem spectral cluster to resolvent intro} requires at least Lipschitz regularity. We will overcome this by localizing to a smaller scale.
\end{remark}

\subsection{Manifolds with boundary}

\begin{theorem}\label{theorem boundary intro}
Let $\Delta_g$ be the Laplace--Beltrami operator on a $n$-dimensional compact Riemannian manifold with boundary, subject to either Dirichlet or Neumann boundary conditions. Assume that $2\leq q\leq 2n/(n-2)$ if $n\geq 3$ and $2\leq q<\infty$ if $n=2$. 
\begin{itemize}
\item[(i)] We have the estimates
\begin{align*}
\|(\Delta_g+(\lambda+\I\lambda^{1/3})^2)^{-1}\|_{L^{q'}\to L^q}\lesssim \langle\lambda\rangle^{2\sigma(q)-1}.
\end{align*}
\item[(ii)] If $n=2,3,4$ and $q\geq 2(n+2)/(n-1)$, then
\begin{align*}
\|(\Delta_g+(\lambda+\I)^2)^{-1}\|_{L^{q'}\to L^q}\lesssim \langle\lambda\rangle^{2\sigma(q)-1}.
\end{align*}
\item[(iii)] If $M$ has strictly concave boundary and the boundary conditions are Dirichlet, then
\begin{align*}
\|(\Delta_g+(\lambda+\I)^2)^{-1}\|_{L^{q'}\to L^q}\lesssim \langle\lambda\rangle^{2\sigma(q)-1}\ln\langle\lambda\rangle.
\end{align*}
\end{itemize}
These estimates are stable under perturbations by real-valued potentials $V\in L^{n/2}(M)$.
\end{theorem}
\begin{remark}
$(ii)$ implies that the uniform Sobolev inequality \eqref{Dos Santos Ferreira, Kenig and Salo} holds for manifolds with boundary in dimensions $n=2,3,4$. The proof is based on a deep result by Smith and Sogge~\cite{MR2316270}, who proved that perfect spectral cluster estimates hold in the stated range. For $n=2$ their result is best possible. The conjectured range for perfect spectral cluster estimates in higher dimensions is $q\geq (6n+4)/(3n-4)$. If true, this would imply the uniform Sobolev inequality in all dimensions $n\geq 3$. The higher dimensional estimates of \cite{MR2316270} are still good enough to obtain this in dimensions $n=3,4$. Their results also yield spectral cluster bounds for $n\geq 5$, $q\geq 4$, but these do not imply uniform Sobolev estimates since $2n/(n-2)<4$.
\end{remark}

\begin{remark}\label{remark proof of theorem boundary (i)}
The estimates for manifolds with boundary are closely related to those for manifolds with Lipschitz metrics.
The reason is that one can double the manifold and reflect the metric across the boundary in an odd (for Dirichlet boundary conditions) or even (for Neumann boundary conditions) way. This produces a manifold without boundary, but with Lipschitz metric, which accounts for the simple bound in $(i)$.
\end{remark}

\begin{remark}\label{remark proof of theorem boundary (iii)}
The fact that the bounds in \cite{MR2316270} were proved by reducing to a manifold without boundary is important here since Theorem \ref{theorem spectral cluster to resolvent intro} does not directly apply to manifolds with boundary. Such a generalization would imply perfect resolvent estimates for manifolds with strictly concave boundaries, based on the corresponding spectral cluster bounds of Grieser~\cite{MR2688254} (in two dimensions) or those of Smith and Sogge~\cite{MR1308407} (in any dimension). Microlocally, the proof of Theorem \ref{theorem spectral cluster to resolvent intro} does not work at ``glancing" points of $T^*(\partial M)$, see \cite[Chapter 24]{MR2304165}. It would be interesting to prove a version of Theorem \ref{theorem spectral cluster to resolvent intro} for manifolds with boundary. However, the abstract bound \eqref{abstract resolvent bound with log loss} holds irrespective of boundaries and yield a weaker bound with a logarithmic loss in $(iii)$. 
\end{remark}

\begin{remark}
The Lipschitz singularity produced by the doubling procedure is of a special type, which is why the bounds of \cite{MR2316270} yield uniform Sobolev inequalities in higher dimensions, while those of \cite{MR3282983} do not.
Incidentally, the proof in \cite{MR2316270} yields the conjectured results for the special Lipschitz singularities considered there. This explains the similarity of the conjectured estimates for Lipschitz metrics and the actual estimates in $(ii)$ for manifolds with boundary.
\end{remark}

\subsection{Improved estimates}
Following Canzani--Galkowski \cite{Canzani--Galkowski}, we denote by ${\Xi}$ the collection of maximal unit speed geodesics for $(M,g)$ and define, for $x\in M$,
$$
\mc{C}_{_{x}}^{r,t}:=\big\{\gamma(t):\,\gamma\in \Xi,\, \gamma(0)=x,\,\exists\, n-1 \text{ conjugate points to } x \text{ in }\gamma(t-r,t+r)\big\},
$$
where conjugate points are counted with multiplicity. 
\begin{theorem}\label{theorem improved intro}
Let $\Delta_g$ be the Laplace--Beltrami operator on an $n$-dimensional compact boundaryless Riemannian manifold with smooth metric. Assume that $q_n< q\leq 2n/(n-2)$ if $n\geq 3$ and $q_n< q<\infty$ if $n=2$. 
Assume that there exist $t_0>0$ and $a>0$ such that 
$$
\inf_{x_1,x_2\in M}d\big(x_1, \mc{C}_{x_2}^{r_t,t}\big)\geq r_t,\qquad\text{ for } t\geq t_0,
$$
with $r_t=\frac{1}{a}e^{-at}.$ Then we have the improved resolvent estimate
\begin{align*}
\|(\Delta_g+(\lambda+\I\epsilon(\lambda))^2)^{-1}\|_{L^{q'}\to L^q}\lesssim \langle\lambda\rangle^{2\sigma(q)-1},
\end{align*}
where $\epsilon(\lambda)=1/\ln\langle\lambda\rangle$. These estimates are stable under perturbations by real-valued potentials $V\in L^{n/2}(M)$.
\end{theorem}

\begin{remark}
For spectral clusters, Canzani-Galkowski \cite[Theorem 2]{Canzani--Galkowski} prove a more general result that does not require global geometric assumptions on $(M,g)$, but only purely dynamical assumptions. In order to avoid lengthy definitions we stated the resolvent estimates for the simpler version \cite[Theorem 1]{Canzani--Galkowski}, but the corresponding result for the general version is also valid. It is interesting to note that, similar to our method, the techniques of \cite{Canzani--Galkowski} do not require long-time wave parametrices.
\end{remark}

\begin{remark}
Since \cite[Theorem 1]{Canzani--Galkowski} includes as a special case manifolds without conjugate points, it generalizes the results of Hassell and Tacy \cite{MR3341481}, where logarithmic improvements for $q>q_n$ were proved for manifolds with non-positive curvature. Blair, Huang, Sire and Sogge \cite[Theorem 1.3]{blair2020uniform} used these to obtain uniform Sobolev inequalities for Schrödinger operators with singular potentials on such manifolds. Theorem \ref{theorem improved intro} generalizes the latter result (in a limited range of exponents) in the same way.
\end{remark}

\begin{remark}
As before, one could interpolate with the trivial $L^2\to L^2$ bound to get resolvent estimates for the full range of $q$. Then one would get a small loss $\epsilon(\lambda)^{-\nu}$, for arbitrary $\nu>0$, at the critical exponent $q_n$.
\end{remark}

\subsection{Fractional Schrödinger operators}

\begin{theorem}\label{theorem fractional intro}
Let $(M,g)$ be an $n$-dimensional compact boundaryless Riemannian manifold, $n\geq 2$, with $C^{1,1}$ metric. Assume that $2n/(n+1)\leq \alpha \leq n$ and $2\leq q\leq 2n/(n-\alpha)$ if $n>\alpha$ or $2\leq q <\infty$ if $n=\alpha$. If
$V\in L^{n/\alpha}(M)$, then
\begin{align*}
\|((-\Delta_g)^{\alpha/2}+V-(\lambda+\I)^{\alpha})^{-1}\|_{L^{q'}\to L^q}\lesssim \langle\lambda\rangle^{2\sigma(q)+1-\alpha}.
\end{align*}
\end{theorem}

\begin{corollary}\label{corollary fractional intro}
If $H_V=(-\Delta_g)^{\alpha/2}+V\geq 0$ 
and $\Pi^V_{[\lambda,\lambda+1]}=\mathbf{1}_{[\lambda,\lambda+1]}(H_V^{1/\alpha})$, then
\begin{align*}
\|\Pi^V_{[\lambda,\lambda+1]}\|_{L^2\to L^q}\lesssim\langle\lambda\rangle^{\sigma(q)}
\end{align*}
holds for $2\leq q\leq 2n/(n-\alpha)$ if $n>\alpha$ or $2\leq q <\infty$ if $n=\alpha$.
\end{corollary}

\begin{remark}
Note that, under the assumption of Theorem~\ref{theorem fractional intro}, $H_V$ is bounded from below. Hence, after possibly adding a constant, we can always assume $H_V\geq 0$.
Here the operator $(-\Delta_g)^{\alpha/2}$ is the so-called spectral fractional Laplacian, defined by the functional calculus. Thus the spectral cluster bounds for $V=0$ coincide with those for the Laplace-Beltrami operator. 
\end{remark}

\begin{remark}
Huang, Sire and Zhang \cite{huang2020lp} proved the spectral cluster estimate in Corollary~\ref{corollary fractional intro} in the case of smooth metric for a larger range of exponents $q$. Here we only use resolvent estimates as input, which limits the range to the stated one. Estimates for higher exponents could be derived by using heat kernel bounds as in \cite{huang2020lp}. For larger $q$ the condition $V\in L^{n/\alpha}(M)$ is not enough, but the spectral cluster bounds in \cite{huang2020lp} hold under the additional assumption that $V$ is in the Kato class of order $\alpha$, see also \cite{blair2019quasimode,blair2020uniform,frank2020sharp} for similar recent results involving critically singular potentials. 
\end{remark}

\begin{remark}
Although Theorem \ref{theorem fractional intro} does not quite follow from Theorem \ref{theorem spectral cluster to resolvent intro}, a slight modification of its proof suffices. In fact, Theorem \ref{theorem spectral cluster to resolvent intro} could be stated for pseudodifferential operators of real principal type. We will leave such generalizations to a future investigation. 
\end{remark}

\begin{remark}
Related to the previous remark, the strategy for the fractional Laplacian could also be adapted to simplify and generalize result of Krupchyk and Uhlmann \cite{MR3285241} for higher order differential operators. Here one assumes that the fibers of the unit cosphere bundle associated to the principal symbol are strictly convex.
Once more, this could also be deduced from the results of Koch and Tataru \cite[Theorem 2.5]{MR2094851}. 
\end{remark}

\subsection{Overview of results}

Figure 1 summarizes the above results.

\begin{figure}[h]
\begin{center}
\begin{tabular}{ |c|c|c|c| }
\hline
Operator & Regularity/boundary  & Spectral cluster & Resolvent \\ 
\hline
$-\Delta_g$, & $g\in C^{\infty}$ & \cite{MR930395} &
\cite{MR3200351}, \cite{BourgainShaoSoggeEtAl2015}, \cite{MR3285241}, \cite{MR3848231}, \cite{MR3620715} \\
\hline
&  & \cite{Canzani--Galkowski} & Theorem \ref{theorem improved intro} \\
\hline
 & $g\in C^{s}$ & \cite{MR2262171,MR2280790,MR2289621,MR3282983} & Theorem \ref{theorem nonsmooth into}  \\
 \hline
  & $\partial M\neq \emptyset$ & \cite{MR2316270}, \cite{MR1308407} & Theorems \ref{theorem boundary intro} \\
\hline
$-\Delta_g+V$ & $V\in L^{n/2}$ & \cite{blair2019quasimode}, \cite{blair2020uniform} & \cite{blair2020uniform}, Theorems \ref{theorem nonsmooth into}, \ref{theorem boundary intro} \\
\hline
$(-\Delta_g)^{\alpha/2}+V$ & $V\in L^{n/\alpha}$ & \cite{huang2020lp} & Theorem \ref{theorem fractional intro} \\
\hline
\end{tabular}
\caption{Summary of results.}
\end{center}
\end{figure}

\subsection{Outline of the paper}
In Section \ref{Section Abstract results} we prove, in a rather abstract setting, that sharp spectral cluster estimates imply resolvent estimates that are sharp up to a logarithmic loss in $\lambda$. In Section \ref{Section From spectral cluster to resolvent bounds} we prove Theorem \ref{theorem spectral cluster to resolvent intro} and Theorem \ref{theorem stability under perturbations non-effective intro}. In Section \ref{Section Applications} we then show how Theorems \ref{theorem nonsmooth into}, \ref{theorem boundary intro} and \ref{theorem fractional intro} follow from a combination of these results and known spectral cluster estimates. We will also discuss simple proofs of known resolvent bounds on compact manifolds and on Euclidean space. In the latter case, we show that the Kenig--Ruiz--Sogge bounds directly from the Tomas--Stein theorem.

\subsection{Notation} If not indicated otherwise (and with the exception of Section \ref{Section Abstract results} and Subsections \ref{subsec. Localization and reduction to an equation on Rn}--\ref{subs. proof eps>1}), $L^q$ always stands for $L^q(M)$, with respect to the natural volume form induced by the Riemannian metric. The norm of an operator $A:L^p\to L^q$, initially defined on $L^2\cap L^p$ and extended by continuity, is denoted by $\|A\|_{L^p\to L^q}$. We write $2^*$ for the Sobolev exponent $2n/(n-2)$ if $n\geq 3$ and any fixed, sufficiently large number $2^*<\infty$ for $n=2$. We will call a bump function adapted to a set if it is supported on that set and equal to one on a slightly smaller dilate of the same set. $a\lesssim b$ means $a\leq Cb$ for some universal constant $C$ (independent $\lambda$). If the constant is assumed to be small, we write $a\ll b$. We sometimes abbreviate statements like $a\leq C_N \lambda^{-N}b$ for any $N>0$ by $a\lesssim \lambda^{-N}b$.
For $x\in\R^n$, we write $D_x=-\I\partial_x$. The Weyl quantization of a pseudodifferential symbol $a$ is denoted by $a^w$.

\section{Abstract results}\label{Section Abstract results}
We consider the following setting: Let $X$ be a $\sigma$-finite measure space and let $A$ be a self-adjoint (possibly unbounded) operator on $L^2(X)$. The Lebesgue spaces $L^p(X)$ are \emph{compatible} as $p$ varies in $[1,\infty]$, in the sense that $L^{p_1}(X)\cap L^{p_2}(X)$ is dense in both spaces in the intersection and complete with respect to the norm $\|\cdot\|_{L^{p_1}\cap L^{p_2}}=\|\cdot\|_{L^{p_1}}+\|\cdot\|_{L^{p_2}}$ (see \cite[Problem 2.2.9]{MR2359869}). We fix $q\in (2,\infty)$ and assume that the resolvent operators $R(z,A)=(A-z)^{-1}$, $z\in\C\setminus\R$, have a \emph{consistent} extension as bounded operators $L^{q'}(X)\to L^2(X)$ and $L^{q'}(X)\to L^q(X)$. Consistency means that $R(z,A)$ coincides with the $L^2$-resolvent operator on the intersection $L^2(X)\cap L^{q'}(X)$ (see \cite{MR2359869}). By the Stone-Weierstrass theorem this implies consistency of $m(A)$ for arbitrary bounded Borel functions $m$ (see e.g.\ \cite[Lemma 4.4]{MR3243083}). We will use the notation $\Pi_{[a,b]}:=\mathbf{1}_{[a,b]}(A)$ for $a,b\in\R$, $a<b$.
In this section $L^q$ will stand for $L^q(X)$.

We start with a spectral multiplier estimate.

\begin{lemma}\label{lemma spectral multiplier bound}
Let $m_1,m_2:[0,\infty)\to \C$ be bounded Borel functions and set $m:=m_1m_2$. 
Further, let $2\leq q\leq \infty$, $J\subset\R$ an interval, and let $(\tau_k)_{k=0}^{N+1}$ be a partition of $J$. Assume that
\begin{align}
&\sup_{0\leq k\leq N}\|\Pi_{[\tau_k,\tau_{k+1}]}\|_{L^{q'}\to L^2}<\infty,\\
&M_j^2:=\sum_{k=0}^N\sup_{\tau\in [\tau_k,\tau_{k+1}]}|m_j(\tau)|^2\|\Pi_{[\tau_k,\tau_{k+1}]}\|^2_{L^{q'}\to L^2}<\infty,\label{def. Mj}
\end{align}
for $j=1,2$. Then we have the spectral multiplier estimates
\begin{align}\label{spectral multiplier bound}
\|\mathbf{1}_J(A) m_j(A)\|_{L^{q'}\to L^{2}}\leq M_j,\quad
\|\mathbf{1}_J(A) m(A)\|_{L^{q'}\to L^{q}}\leq M_1M_2.
\end{align}
\end{lemma}

\begin{remark}
Lemma \ref{lemma spectral multiplier bound} is similar to \cite[Lemma 2.3]{BourgainShaoSoggeEtAl2015}, but improves upon the latter when applied to the spectrally localized resolvent. The reason is that it uses the $TT^*$ argument in a more efficient way and avoids the use of the triangle inequality. 
\end{remark}

\begin{proof}
For the proof we set $\Pi_k:=\Pi_{[\tau_k,\tau_{k+1}]}$.
In the following we only consider functions $u\in \Ran \mathbf{1}_J(A)$. For $j=1,2$ and $\delta>0$ we define the spectral multipliers
\begin{align*}
S_{j,\delta} u:=\sum_{k=0}^{N}(\sup_{\tau\in [\tau_k,\tau_{k+1}]}|m_j(\tau)|+\delta)^{-1}m_j(A)\Pi_{k} u.
\end{align*}
By the spectral theorem\footnote{We will sometimes use the phrase ``by orthogonality" or ``by the functional calculus" for arguments relating to the spectral theorem.}, we have $\|S_{j,\delta}\|_{L^2\to L^2}\leq 1$. 
By orthogonality we obtain the bound
\begin{align*}
\|m_j(A)u\|_{L^2}^2&=\sum_{k=0}^{N}\sup_{\tau\in [\tau_k,\tau_{k+1}]}(|m_j(\tau)|+\delta)^2\|\Pi_{k}(S_{j,\delta}u)\|_{L^2}^2\\
&\leq \sum_{k=0}^{N}\sup_{\tau\in [\tau_k,\tau_{k+1}]}(|m_j(\tau)|+\delta)^2\|\Pi_{k}\|_{L^{q'}\to L^2}^2 \|u\|_{L^{q'}}^2
\end{align*}
for $u\in L^2\cap L^{q'}$. Taking $\delta\to 0$, the first bound in \eqref{spectral multiplier bound} follows by dominated convergence and density of $L^2\cap L^{q'}$ in $L^{q'}$. Since the assumption \eqref{def. Mj} is invariant under complex conjugation of $m_j$, we obtain the same bound for $m_j(A)^*$. 
The Schwarz inequality then yields
\begin{align*}
|\langle m(A)u,v \rangle|=|\langle m_2(A)u,m_1(A)^*v \rangle|\leq M_1M_2\|u\|_{L^{q'}}\|v\|_{L^{q'}}.
\end{align*}
Hence, the second bound in \eqref{spectral multiplier bound} follows from duality of $L^q(X)$ and $L^{q'}(X)$ (see e.g.\ \cite[Theorem 2.1.10]{MR2359869}).
\end{proof}

\begin{proposition}\label{proposition abstract bounds}
Let $0<\epsilon\leq \lambda$. Then the following estimates hold.
\begin{align}\label{L2Lq resolvent implies L2Lq spectral cluster}
\|\Pi_{[\lambda,\lambda+\epsilon]}\|_{L^{q'}\to L^{2}}\lesssim \epsilon\lambda\|\Pi_{[0,2\lambda]}\im(A^2-(\lambda+\I\epsilon)^2)^{-1}\|_{L^{q'}\to L^2},
\end{align}

\begin{align}\label{Lq'L2 spectral cluster implies L2Lq resolvent}
\|\Pi_{[0,2\lambda]}(A^2-(\lambda+\I\epsilon)^2)^{-1}\|_{L^{q'}\to L^{2}}\lesssim (\epsilon\lambda)^{-1}\sup_{k}\|\Pi_{[\epsilon k,\epsilon(k+1)]}\|_{L^{q'}\to L^2},
\end{align}

\begin{align}\label{Lq'Lq resolvent implies L2Lq spectral cluster}
\|\Pi_{[\lambda,\lambda+\epsilon]}\|^2_{L^{q'}\to L^{2}}\lesssim \epsilon\lambda\|\Pi_{[0,2\lambda]}\im(A^2-(\lambda+\I\epsilon)^2)^{-1}\|_{L^{q'}\to L^q},
\end{align}

\begin{align}\label{Lq'L2 spectral cluster implies Lq'Lq resolvent}
\|\Pi_{[0,2\lambda]}(A^2-(\lambda+\I\epsilon)^2)^{-1}\|_{L^{q'}\to L^{q}}\lesssim (\epsilon\lambda)^{-1}\ln\langle\lambda/\epsilon\rangle\sup_{k}\|\Pi_{[\epsilon k,\epsilon(k+1)]}\|^2_{L^{q'}\to L^2},
\end{align}

For $\mu\gtrsim \epsilon$ and $\beta>1$ we also have

\begin{align}\label{Lq'L2 spectral cluster implies Lq'Lq resolvent (stronger version)}
\|\Pi_{[0,2\lambda]}(A^2-(\lambda+\I\mu)^2)^{-1}\|_{L^{q'}\to L^{q}}\lesssim (\epsilon\lambda)^{-1}\langle\mu/\lambda\rangle^{-1}\ln\langle\lambda/\mu\rangle\sup_{k}\|\Pi_{[\epsilon k,\epsilon(k+1)]}\|^2_{L^{q'}\to L^2},
\end{align}

\begin{align}\label{Lq'L2 spectral cluster implies Lq'Lq resolvent with powers >1}
\|\Pi_{[0,2\lambda]}(A^2-(\lambda+\I\mu)^2)^{-\beta}\|_{L^{q'}\to L^{q}}\lesssim_{\beta} (\epsilon\lambda)^{-\beta}(\epsilon/\mu)^{\beta-1}\langle\mu/\lambda\rangle^{-\beta}\sup_{k}\|\Pi_{[\epsilon k,\epsilon(k+1)]}\|^2_{L^{q'}\to L^2}.
\end{align}

The suprema are taken over $k\in\Z^+$ with $k\leq 2\lambda/\epsilon$.
\end{proposition}

\begin{remark}
The relationship between $L^{q'}\to L^2$ spectral cluster and $L^{q'}\to L^2$ resolvent estimates is well known. The bound \eqref{Lq'Lq resolvent implies L2Lq spectral cluster} is contained in \cite[Lemma 10]{MR3620715}.
Estimates similar to \eqref{Lq'L2 spectral cluster implies Lq'Lq resolvent with powers >1} may be found e.g.\ in \cite{KochTataru2009} for the Hermite operator. The bounds \eqref{Lq'L2 spectral cluster implies Lq'Lq resolvent}--\eqref{Lq'L2 spectral cluster implies Lq'Lq resolvent (stronger version)} seem to be new.
\end{remark}

\begin{proof}
To prove \eqref{L2Lq resolvent implies L2Lq spectral cluster} we use that
\begin{align}\label{imaginary part of resolvent bounds spectral projection from above}
\im(\tau^2-(\lambda+\I\epsilon)^2)^{-1}\gtrsim (\epsilon\lambda)^{-1}\mathbf{1}_{[\lambda,\lambda+\epsilon]}(\tau).
\end{align}
By the functional calculus it follows that, for $u\in L^2\cap L^{q'}$
\begin{align}\label{lower bound for imaginary part of the resolvent}
\|\Pi_{[\lambda,\lambda+\epsilon]}u\|_{L^{2}}\lesssim \epsilon\lambda\|\Pi_{[0,2\lambda]}\im(A^2-(\lambda+\I\epsilon)^2)^{-1}u\|_{L^2}.
\end{align}
Hence, \eqref{L2Lq resolvent implies L2Lq spectral cluster} again follows by density of $L^2\cap L^{q'}$ in $L^{q'}$. 
The same argument, together with the fact that $\Pi_{[\lambda,\lambda+\epsilon]}^2=\Pi_{[\lambda,\lambda+\epsilon]}$, yields
\begin{align*}
\|\Pi_{[\lambda,\lambda+\epsilon]}u\|^2_{L^2}=\langle \Pi_{[\lambda,\lambda+\epsilon]}u,u\rangle\lesssim \epsilon\lambda\langle \Pi_{[0,2\lambda]} \im(A^2-(\lambda+\I\epsilon)^2)^{-1}u,u\rangle,
\end{align*}
and this implies \eqref{Lq'Lq resolvent implies L2Lq spectral cluster}. Inequality \eqref{Lq'L2 spectral cluster implies Lq'Lq resolvent} is a special case of \eqref{Lq'L2 spectral cluster implies Lq'Lq resolvent (stronger version)} for $\mu=\epsilon$. To prove \eqref{Lq'L2 spectral cluster implies L2Lq resolvent}, \eqref{Lq'L2 spectral cluster implies Lq'Lq resolvent (stronger version)}, \eqref{Lq'L2 spectral cluster implies Lq'Lq resolvent with powers >1} we use Lemma~ \ref{lemma spectral multiplier bound} with $J=[0,2\lambda]$, $N=\lceil 2\lambda/\epsilon\rceil$, $\tau_k=\epsilon k$, and with
\begin{align}\label{mj for resolvent}
 m_j(\tau)=(\tau^2-(\lambda+\I\mu)^2))^{-\alpha},\quad \alpha\in \{1,1/2,\beta/2\}
\end{align}
for $j=1,2$. We will show that $m_j$ is smooth on the $\epsilon$-scale and bound the sum \eqref{def. Mj} in the definition of $M_j$ by a Darboux (Riemann) integral. The required smoothness is captured by the bound 
\begin{align}\label{smoothness on the epsilon-scale}
\sup_{\tau\in [\epsilon k,\epsilon (k+1)]}|m_j(\tau)|^2\lesssim \inf_{\tau\in [\epsilon k,\epsilon (k+1)]}|m_j(\tau)|^2,
\end{align}
with an implied constant that is uniform in $\lambda,\epsilon,\mu$. 
This holds since, for $\tau\in[\epsilon k,\epsilon (k+1)]$, we have $|\tau+(\lambda+\I\mu)|\approx \lambda+\mu$ and $|\tau-(\lambda+\I\mu)|\approx |\epsilon k-\lambda|+\mu$. The first bound is obvious, and so is the upper bound for the second. We also have the lower bound $||\epsilon k-\lambda|-\epsilon|+\mu$ for the second. We only need to consider the case $|\epsilon k-\lambda|\geq C\mu$ since the other case is again obvious. If $\mu\geq C'\epsilon$, we get $||\epsilon k-\lambda|-\epsilon|\geq |\epsilon k-\lambda|(1-1/(CC'))$. Note that $C'$ is fixed, and we may choose $C=1/(2C')$, establishing the desired lower bound.

Consider the partition $P$ of $[0,(N+1)\epsilon]\supset [0,2\lambda]$ given by $\tau_0<\tau_1\ldots<\tau_{N+1}$ and define the Darboux sums
\begin{align*}
s(P)=\epsilon\sum_{k=0}^N\inf_{\tau\in [\epsilon k,\epsilon (k+1)]}|m_j(\tau)|^2,\quad
S(P)=\epsilon\sum_{k=0}^N\sup_{\tau\in [\epsilon k,\epsilon (k+1)]}|m_j(\tau)|^2.
\end{align*}
Obviously, we have
\begin{align}\label{s(P)<...<S(P)}
s(P)\leq \int_0^{(N+1)\epsilon}|m_j(\tau)|^2\rd\tau\leq S(P).
\end{align}
On the other hand, by \eqref{smoothness on the epsilon-scale}, we also have the converse inequalities, up to constants. Hence,
\begin{align*}
S(P)\lesssim\int_0^{4\lambda}|\tau^2-(\lambda+\I\mu)^2)|^{-2\alpha}\rd\tau \lesssim (\lambda+\mu)^{-2\alpha}\mu^{1-2\alpha} 
\ln^{\nu}\langle\lambda/\mu\rangle,
\end{align*}
with $\nu=1$ if $\alpha=1/2$ and $\nu=0$ if $\alpha>1/2$. 
Since, by definition of $M_j$ and $S(P)$, 
\begin{align}\label{Upper bound Mj^2}
M_j^2\leq \epsilon^{-1}S(P)\sup_{1\leq k\leq N}\|\Pi_{k}\|^2_{L^{q'}\to L^2},
\end{align}
the claims follows from \eqref{spectral multiplier bound}.
\end{proof}

In most applications, the spectral localization can easily be removed by Sobolev embedding and elliptic estimates. Moreover, the spectral cluster bounds usually take the form \eqref{abstract spectral cluster bound epsilon}.
For convenience we record the following corollary.  From now on we will use the notation $\Pi_{\leq 2\lambda}:=\Pi_{[0,2\lambda]}$.

\begin{corollary}\label{corollary equivalence cluster and quasimode}
Let $\epsilon=\epsilon(\lambda)$ and $\delta=\delta(\lambda)$ be such that $0<\delta,\epsilon\leq \lambda$. The following bounds are equivalent.
\begin{itemize}
\item[(a)] The spectral cluster bound: 
\begin{align}\label{abstract spectral cluster bound epsilon}
\|\Pi_{[\lambda,\lambda+\epsilon]}\|_{L^2\to L^{q}}\lesssim \delta\langle\lambda\rangle^{\gamma(q)}.
\end{align}
\item[(b)] The spectrally localized $L^2\to L^q$ resolvent bound:
\begin{align}\label{abstract L2Lq resolvent bound}
\|\Pi_{\leq 2\lambda}(A^2-(\lambda+\I\epsilon)^2)^{-1}\|_{L^{2}\to L^{q}}\lesssim \epsilon^{-1}\delta\langle\lambda\rangle^{\gamma(q)-1}.
\end{align}
\item[(c)] The quasimode bound:
\begin{align}\label{abstract quasimode bound}
\|u\|_{L^{q}}\lesssim \delta\langle\lambda\rangle^{\gamma(q)}\|u\|_{L^{2}}+\epsilon^{-1}\delta\langle\lambda\rangle^{\gamma(q)-1}\|(A^2-\lambda^2)u\|_{L^{2}},
\end{align}
for all functions $u$ spectrally localized in $[0,2\lambda]$.
\end{itemize}
Moreover, if $\mu=\mu(\lambda)$, $\mu\gtrsim\epsilon$, these bounds imply the spectral cluster estimate 
\begin{align}\label{abstract spectral cluster bound longer window mu}
\|\Pi_{[\lambda,\lambda+\mu]}\|_{L^2\to L^{q}}\lesssim (\mu/\epsilon)^{1/2} \delta\langle\lambda\rangle^{\gamma(q)}
\end{align}
and the spectrally localized $L^{q'}\to L^q$ resolvent estimate
\begin{align}\label{abstract resolvent bound with log loss}
\|\Pi_{\leq 2\Lambda}(A^2-(\lambda+\I\mu)^2)^{-1}\|_{L^{q'}\to L^{q}}\lesssim \epsilon^{-1}\langle\mu/\lambda\rangle^{-1}\ln\langle\lambda/\mu\rangle \delta^2\lambda^{2\gamma(q)-1}.
\end{align}
\end{corollary}

\begin{proof}
The equivalence of $(a)$, $(b)$ and \eqref{abstract resolvent bound with log loss} follow from Proposition \ref{proposition abstract bounds}. The implication $(b)\implies (c)$ follows from the triangle inequality, and $(c)\implies (a)$ follows from 
\begin{align*}
\|(A^2-\lambda^2)\Pi_{[\lambda,\lambda+\epsilon]}u\|_{L^{2}}\lesssim\epsilon\lambda\|u\|_{L^{2}}.
\end{align*}
The spectral cluster estimate \eqref{abstract spectral cluster bound longer window mu} follows from \eqref{abstract spectral cluster bound epsilon} and orthogonality. 
\end{proof}

\begin{remark}\label{remark elliptic estimates outside a cone}
The bound \eqref{Lq'Lq resolvent implies L2Lq spectral cluster} shows that if the spectral cluster estimate \eqref{abstract spectral cluster bound epsilon} is sharp, then the resolvent estimate \eqref{abstract resolvent bound with log loss} with $\mu=\epsilon$ will be sharp up to a logarithm. It also shows that there is a separation of scales at $\mu=\lambda$, at which point the logarithm becomes bounded. This corresponds to an elliptic estimate outside a cone. It is slightly less obvious that \eqref{abstract resolvent bound with log loss} is almost sharp for intermediate values $\epsilon\lesssim\mu\lesssim\lambda$. To see this, one uses orthogonality and the assumed sharpness of \eqref{abstract spectral cluster bound epsilon} to get a lower bound
\begin{align*}
\|\Pi_{[\lambda,\lambda+\mu]}\|_{L^2\to L^{q}}\gtrsim (\mu/\epsilon)^{1/2} \lambda^{\gamma(q)}.
\end{align*}
Then the same argument leading to \eqref{lower bound for imaginary part of the resolvent} with $\mu$ in place of $\epsilon$ yields the conclusion.
\end{remark}

\begin{remark}
In most applications $\delta(\lambda)$ will be a power of $\lambda$ and hence could be omitted by redefining $\gamma(q)$. However, in some cases $\delta(\lambda)=\ln^{\nu}\langle\lambda\rangle$ for some $\nu\in\R$, see e.g.\ \eqref{Lq'Lq Lipschitz intro q=8}.
For the Laplace--Beltrami operator on a compact boundaryless manifold, corresponding to $A=\sqrt{-\Delta_g}$, the relevant ranges for smooth metrics are $\epsilon\in [\lambda^{-1},1]$ and $\delta\in [\epsilon^{1/2},1]$, with $\delta(\lambda),\epsilon(\lambda)$ monotonically decreasing \cite{huang2020lp}. The universal estimates of Sogge~\cite{MR930395} correspond to $\delta,\epsilon=1$ and $\gamma(q)=\sigma(q)$. In the Euclidean case, the Tomas--Stein bound \eqref{Tomas--Stein spectral cluster form} corresponds to $\delta=\epsilon^{1/2}$ and $\epsilon>0$.
\end{remark}

\section{From spectral cluster to resolvent bounds}
\label{Section From spectral cluster to resolvent bounds}
We now consider a setting where the logarithmic loss in \eqref{abstract resolvent bound with log loss} can be removed. We assume that $X$ is a smooth compact Riemannian manifold $(M,g)$ without boundary and $A=\sqrt{-\Delta_g}$. The methods we use require Lipschitz regularity of the metric. This condition is natural for the paradifferential reductions to a first order frequency localized equation (see \cite{MR2280790,MR2094851} and especially Step 1 in the proof of \cite[Theorem 2.1]{MR3282983}). 

\subsection{A generalization of Theorem \ref{theorem spectral cluster to resolvent intro}}
We generalize the result of Theorem \ref{theorem spectral cluster to resolvent intro} to spectral clusters of any size $\epsilon\geq \lambda^{-1}$ (not just $\epsilon=1$).

\begin{theorem}\label{theorem spectral cluster to resolvent stronger version}
Let $(M,g)$ be a compact Riemannian manifold without boundary and with Lipschitz metric $g$. Assume that $2\leq q\leq 2^*$, and that the spectral cluster estimate 
\begin{align}\label{spectral cluster estimate strong theorem}
    \|\Pi_{[\lambda,\lambda+\epsilon]}\|_{L^{2}\to L^{q}}\lesssim \delta\langle\lambda\rangle^{\gamma(q)}
\end{align}
holds for some $\gamma(q)\in[\sigma(q),\sigma(q)+1/2]$,  $\epsilon\geq \lambda^{-1}$, $\delta\geq\epsilon^{1/2}$. Then
\begin{align}\label{resolvent estimate strong theorem}
\|(\Delta_g+(\lambda+\I \epsilon)^2)^{-1}\|_{L^{q'}\to L^{q}}\lesssim \epsilon^{-1}\delta^2\langle\lambda\rangle^{2\gamma(q)-1}.
\end{align}
\end{theorem}

\begin{remark}
    The theorem shows that improved spectral cluster bounds ($\epsilon=o(1)$ as $\lambda\to\infty$) imply improved resolvent bounds. For smooth metrics, this was shown by Bourgain, Shao, Sogge and Yao \cite[Theorem 1.3]{BourgainShaoSoggeEtAl2015} for the exponent $q=2^*$. The proof uses the Hadamard parametrix for the wave equation, which is not available under our limited smoothness assumptions. We therefore include the argument for improved bounds, avoiding the Hadamard parametrix and extending the result of \cite{BourgainShaoSoggeEtAl2015} to all $2\leq q\leq 2^*$.
\end{remark}
We briefly outline the plan of the proof. We need to prove that if $u$ satisfies the resolvent equation 
\begin{align}\label{resolvent eq. added revision}
    (\Delta_g+(\lambda+\I\epsilon)^2)u=F,
\end{align}
and \eqref{spectral cluster estimate strong theorem} holds, then
\begin{align}\label{resolvent estimate added revision}
    \|u\|_{L^q(M)}\lesssim \epsilon^{-1}\delta^2\langle\lambda\rangle^{2\gamma(q)-1} \|F\|_{L^q(M)}.
\end{align}
 Subsections \ref{subs. reductions}--\ref{subs. proof eps>1} are devoted to the proof of the case $\epsilon\geq 1$. Subsection \ref{subs. eps<1} deals with the proof for $\lambda^{-1}\leq\epsilon<1$.

\subsection{Reductions}\label{subs. reductions}
Until the end of Subsection \ref{subs. proof eps>1} we will consider the case $\epsilon\geq 1$. Note that, by Sobolev embedding, the estimate is trivial for $\lambda\leq 1$, so we will always assume $\lambda\geq 1$ in the following. 
 We will perform a series of reductions:
 \begin{itemize}
     \item[(i)] Reduction to the case $\epsilon\leq \lambda$ and to spectrally localized solutions ($u=\Pi_{\leq 2\lambda}u$) via elliptic estimates;
     \item[(ii)] First application of the spectral cluster bound;
     \item[(iii)] Localization and reduction to an equation on $\R^n$;
     \item[(iv)] Symbol smoothing and factorization; 
     \item[(v)] Reduction to a first order hyperbolic (half wave) equation;
     \item[(vi)] Reduction of the inhomogeneous to a homogeneous estimate (second application of the spectral cluster bound).
 \end{itemize}

\subsection{Elliptic estimates}\label{subs. Elliptic estimates}
As already pointed out in the introduction, stronger estimates than \eqref{resolvent estimate strong theorem} hold in the elliptic regime, that is when
\begin{itemize}
    \item[(a)] either $\epsilon\geq \lambda$, 
    \item[(b)] or $u=\Pi_{>2\lambda}u$.
\end{itemize}
The following proposition thus takes care of the first reduction step.

\begin{proposition}\label{prop. elliptic} Let $(M,g)$ be a compact Riemannian manifold without boundary and with Lipschitz metric. Assume that $u$ satisfies \eqref{resolvent eq. added revision} and that either (a) or (b) hold. Then, for $2\leq q\leq 2^*$, $\lambda\geq 1$, we have
    \begin{align}\label{elliptic estimate}
\|(\Delta_g+(\lambda+\I \epsilon)^2)^{-1}u\|_{L^q}\lesssim \lambda^{2n(\frac{1}{2}-\frac{1}{q})-2}\|u\|_{L^{q'}},
    \end{align}
      \begin{align}\label{elliptic estimate L2Lq}
\|(\Delta_g+(\lambda+\I \epsilon)^2)^{-1}u\|_{L^q}\lesssim \lambda^{n(\frac{1}{2}-\frac{1}{q})-2}\|u\|_{L^{2}}.
    \end{align}  
\end{proposition}

\begin{remark}
    Since $\eps^{-1}\delta^2\geq 1$ and $\gamma(q)\geq n(\frac{1}{2}-\frac{1}{q})-\frac{1}{2}$, we have 
    \begin{align*}
       \lambda^{2n(\frac{1}{2}-\frac{1}{q})-2}\leq  \eps^{-1}\delta^2\lambda^{2\gamma(q)-1}.
    \end{align*}
\end{remark}

\begin{proof}
Taking inner product in $L^2(M)$ of both sides of \eqref{resolvent eq. added revision} against $u$ yields
\begin{align}\label{IBP revision}
    \|\nabla_g u\|_{L^2(M)}^2-(\lambda+\I\epsilon)^2\|u\|_{L^2(M)}^2=-\langle F,u\rangle_{L^2(M)},
\end{align}
where $\nabla_g$ is the gradient with respect to the metric $g$. 
Taking real and imaginary parts, we obtain
\begin{align}
    \|\nabla_g u\|_{L^2}^2-(\lambda^2-\epsilon^2)\|u\|_{L^2}^2&\leq \|u\|_{H^1}\|F\|_{H^{-1}},\label{real part}
    \\
       2\epsilon\lambda \|u\|_{L^2}^2&\leq \|u\|_{H^1}\|F\|_{H^{-1}},\label{im part}
\end{align}
where $\|u\|_{H^1}^2:=\|\nabla_g u\|_{L^2}^2+\|u\|_{L^2}^2$ and $H^{-1}$ is the dual space. In case (a), adding \eqref{real part} and \eqref{im part}, we get
\begin{align}\label{add real and imag. part prelim.}
    \|\nabla_g u\|_{L^2}^2+\lambda^2\|u\|_{L^2}^2\lesssim |\langle F,u\rangle|\leq\|u\|_{H^1}\|F\|_{H^{-1}}\lesssim c\|u\|_{H^1}^2+c^{-1}\|F\|_{H^{-1}}^2,
\end{align}
where $c>0$ is arbitrary. Choosing $c$ sufficiently small, we may absorb the first term on the right of \eqref{add real and imag. part} in the left hand side, so that
\begin{align}\label{add real and imag. part}
   \lambda^2\|u\|_{L^2}^2+ \|u\|_{H^1}\lesssim \|F\|_{H^{-1}},
\end{align}
In case (b), since $\Pi_{>2\lambda}u=u$, we have 
\begin{align*}
  \|\nabla_g u\|_{L^2(M)}^2=\langle \Delta_g u,u\rangle_{L^2(M)}=  \|\sqrt{-\Delta_g} u\|_{L^2(M)}^2\geq 4\lambda^2\| u\|_{L^2(M)}^2.
\end{align*}
This together with \eqref{real part} yields \eqref{add real and imag. part prelim.} and hence \eqref{add real and imag. part}.
Sobolev embedding $H^1(M)\subset L^{\frac{2n}{n-2}}(M)$, the dual embedding $L^{\frac{2n}{n+2}}(M)\subset H^{-1}(M)$ and \eqref{add real and imag. part} yield
\begin{align*}
\|u\|_{L^{\frac{2n}{n-2}}}\lesssim  \|u\|_{H^1}\lesssim \|F\|_{H^{-1}} \lesssim \|F\|_{L^{\frac{2n}{n+2}}},       
\end{align*}
while \eqref{add real and imag. part prelim.} also yields
\begin{align*}
    \lambda^2\|u\|_{L^2}^2\lesssim |\langle F,u\rangle|\leq \|u\|_{L^2}\|F\|_{L^2},
\end{align*}
hence $\|u\|_{L^2}\lesssim \lambda^{-2}\|F\|_{L^2}$. Interpolation between the $L^{\frac{2n}{n+2}}\to L^{\frac{2n}{n-2}}$ and the $L^2\to L^2$ bound yields \eqref{elliptic estimate} for $q$ in the stated range
since
\begin{align*}
    \frac{1}{q}=\frac{(1-\theta)(n-2)}{2n}+\frac{\theta}{2}\iff \theta=-n(\frac{1}{2}-\frac{1}{q})+1.
\end{align*}
By duality, \eqref{elliptic estimate L2Lq} is equivalent to
      \begin{align}\label{elliptic estimate Lq'L2}
\|(\Delta_g+(\lambda+\I \epsilon)^2)^{-1}u\|_{L^2}\lesssim \lambda^{n(\frac{1}{2}-\frac{1}{q})-2}\|u\|_{L^{q'}}.
    \end{align} 
By \eqref{add real and imag. part prelim.}, we have
\begin{align*}
        \lambda^2\|u\|_{L^2}^2\lesssim |\langle F,u\rangle|\leq \|u\|_{L^{q}}\|F\|_{L^{q'}},
\end{align*}
which together with \eqref{elliptic estimate} gives
\begin{align*}
    \lambda^2\|u\|_{L^2}^2\lesssim  \lambda^{2n(\frac{1}{2}-\frac{1}{q})-2} \|F\|_{L^{q'}}^2
\end{align*}
and hence proves \eqref{elliptic estimate Lq'L2}.
\end{proof}

\begin{remark}
    Since we were only working with quadratic forms, it would have been sufficient to assume that the coefficients of the metric are bounded measurable functions.
    \end{remark}

\subsection{First application of the spectral cluster bound}
Given the elliptic estimate~
\eqref{elliptic estimate}, 
it remains to prove 
 \begin{align}\label{spectrally localized resolvent estimate (revision)}
         \|\Pi_{\leq 2\lambda}u\|_{L^q(M)}\lesssim \epsilon^{-1}\delta^2\lambda^{2\gamma(q)-1} \|(\Delta_g+(\lambda+\I\epsilon)^2)u\|_{L^{q'}(M)}.
 \end{align}
The next lemma shows how the spectral cluster bound \eqref{spectral cluster estimate strong theorem} is used to control various errors that will appear later when we localize in position and/or frequency. 
\begin{lemma}\label{lemma first application of spectral cluster}
Assume that the spectral cluster estimate \eqref{spectral cluster estimate strong theorem} holds. Also assume that any solution of the equation
\begin{align}\label{resolvent eq F+G}
       (\Delta_g+(\lambda+\I\epsilon)^2)u=F+G 
\end{align}
on $M$ satisfies the bound
\begin{align}\label{ostensibly weaker bound}
\|\Pi_{\leq 2\lambda}u\|_{L^q}\lesssim  \delta\lambda^{\gamma(q)}
(\|u\|_{L^2}+\lambda^{-1}\|\nabla_gu\|_{L^2}+(\epsilon\lambda)^{-1}\|G\|_{L^{2}}+\epsilon^{-1}\delta\lambda^{\gamma(q)-1}\|F\|_{L^{q'}}).
\end{align}
Then \eqref{spectrally localized resolvent estimate (revision)} holds.
\end{lemma}

\begin{proof}
By linearity it suffices to prove this separately when $G=0$ or $F=0$.
    By the implication $a)\implies b)$ in Corollary \ref{corollary equivalence cluster and quasimode}, \eqref{spectral cluster estimate strong theorem} implies the $L^2\to L^q$ bound
\begin{align*}
    \|\Pi_{\leq 2\lambda}(\Delta_g+(\lambda+\I\epsilon)^2)^{-1}\|_{L^{2}\to L^{q}}\lesssim \epsilon^{-1}\delta\lambda^{\gamma(q)-1}.
\end{align*}
This is still valid if we replace $\lambda+\I\epsilon$ by $\lambda-\I\epsilon$. Hence,
by duality (note that $\Pi_{\leq 2\lambda}$ commutes with the resolvent), we have the $L^{q'}\to L^2$ bound
\begin{align}
    \|\Pi_{\leq 2\lambda}(\Delta_g+(\lambda+\I\epsilon)^2)^{-1}\|_{L^{q'}\to L^{2}}\lesssim \epsilon^{-1}\delta\lambda^{\gamma(q)-1}.
\end{align}
The latter implies that
\begin{align}\label{Lq'L2 bound}
    \|\Pi_{\leq 2\lambda}u\|_{L^2(M)}
    \lesssim \epsilon^{-1}\delta\lambda^{\gamma(q)-1}
    \|(\Delta_g+(\lambda+\I\epsilon)^2)u\|_{L^{q'}(M)}.
\end{align}
Due to the spectral localization, we also have 
\begin{align*}
    \lambda^{-1}\|\nabla_g\Pi_{\leq 2\lambda}u\|_{L^2(M)}\lesssim \|\Pi_{\leq 2\lambda}u\|_{L^2(M)}.
\end{align*}
The last two displays together with \eqref{ostensibly weaker bound} yield
\begin{align*}
    \|\Pi_{\leq 2\lambda}u\|_{L^q}&\lesssim  \delta\lambda^{\gamma(q)}
(\|\Pi_{\leq 2\lambda}u\|_{L^2}+\lambda^{-1}\|\nabla_g\Pi_{\leq 2\lambda}u\|_{L^2}+\epsilon^{-1}\delta\lambda^{\gamma(q)-1}\|F\|_{L^{q'}})\\
&\lesssim \epsilon^{-1}\delta^2\lambda^{2\gamma(q)-1}\|F\|_{L^{q'}}.
\end{align*}
This proves the claim when $G=0$. The case $F=0$ follows immediately from the quasimode estimate \eqref{abstract quasimode bound} (using the implication $a)\implies c)$ in Corollary \ref{corollary equivalence cluster and quasimode}).
\end{proof}

Note that we can absorb $(2\I\lambda\eps-\eps^2)u$ into the driving term $G$ on the right hand side of~\eqref{resolvent eq F+G} without changing the estimate. We will thus consider the equation
    \begin{align}\label{resolvent eq F+G without ieps}
       (\Delta_g+\lambda^2)u=F+G,
\end{align}
and we need to prove \eqref{ostensibly weaker bound}, given \eqref{spectral cluster estimate strong theorem}.


\subsection{Localization and reduction to an equation on $\R^n$}
\label{subsec. Localization and reduction to an equation on Rn}
 We cover $M$ by finitely many charts, $M=\bigcup_{j=1}^K U_j$, and consider a partition of unity $\phi_j\in C_c^{\infty}(U_j)$, $j=1,\ldots, K$, subordinate to this cover. By the triangle inequality, it would thus suffice to prove \eqref{resolvent estimate added revision} for $\phi_j u$.
If $u$ satisfies the resolvent equation \eqref{resolvent eq F+G}, then $\phi u$ satisfies
\begin{align}\label{resolvent equation for phi u}
     (\Delta_g+\lambda^2)(\phi u)=\phi (F+G)+[\Delta_g,\phi]u.   
\end{align}
Since all the terms are supported in a fixed chart $U_j$, we may write $\Delta_g$ in local coordinates,
\begin{align}
\Delta_g=|g|^{-1/2}\partial_ig^{ij}|g|^{1/2}\partial_j,
\end{align}
where $g^{ij}:=(g^{-1})_{ij}$, $|g|:=|\det g|$ and repeated indices are summed over. 
We set $\rho=|g|^{1/2}$, $a^{ij}=|g|^{1/2}g^{ij}$, and we will work in $L^2(\R^n)=L^2(\R^n,\rd x)$. Note that $\rho$ and $(a^{ij})$ are bounded from above and below by a positive constant (the latter in the sense of quadratic forms). Then \eqref{resolvent equation for phi u} can be written as
\begin{align}\label{resolvent equation for phi u II}
    \partial_i(a^{ij}\partial_j(\phi u))+\lambda^2\rho\phi u=\rho\phi (F+G)+\partial_i(a^{ij}u\partial_j\phi)+a^{ij}\partial_i\phi\partial_j u.
\end{align}
At this point we drop the spectral localization and consider $u$ as a function on $\R^n$ with compact support; 
 without loss of generality, we assume that it is supported in the unit cube $Q=[0,1]^n$. By shrinking coordinate patches and rescaling, if necessary, we can always assume that $g^{ij}$ is globally defined and close to the Euclidean metric in the sense that
\begin{align}\label{g close to Euclidean metric}
\|g^{ij}-\delta^{ij}\|_{\mathrm{Lip}}\ll 1,
\end{align}
Then we also have
\begin{align}\label{a,rho close to Euclidean metric}
\|\rho-1\|_{\mathrm{Lip}}+ \|a^{ij}-\delta^{ij}\|_{\mathrm{Lip}}\ll 1.
\end{align}
Here and in the following (until the end of Subsection \ref{subs. proof eps>1}) all norms are taken over $\R^n$ unless stated otherwise.
In particular, there exists $\kappa$ close to $1$ such for all $x,\xi\in\R^n$,
\begin{align}\label{g uniformly elliptic}
  \kappa^{-1} |\xi|^2\leq a^{ij}(x)\xi_i\xi_j\leq \kappa |\xi|^2.
\end{align}
Note that the localization has produced an 
 additional error term (the second term on the right of \eqref{resolvent equation for phi u II}; the third term can be absorbed into $G$)
 of the form $\partial_iG^i$, where $G^i=a^{ij}u\partial_j\phi$ satisfies $\sum_{i=1}^n\|G^i\|_{L^2}^2\lesssim \|u\|_{L^2}^2$. We will include this into the driving term, that is, we will replace $F+G$ by $F+G+\partial_iG^i$ below. We will write $\Vec{G}=(G_1,\ldots,G_n)$ (so that $\partial_iG^i=\Div \Vec{G}$) and $\|\Vec{G}\|_{L^2}^2=\sum_{i=1}^n\|G^i\|_{L^2}^2$.
Thus \eqref{ostensibly weaker bound} would be a consequence of the following global estimate.

 \begin{proposition}\label{Prop. Rn}
     Assume that $a^{ij}$, $\rho$ are globally defined on $\R^n$ and satisfy \eqref{a,rho close to Euclidean metric}. If
     \begin{align}\label{eq. Prop. Rn}
    \partial_i(a^{ij}\partial_ju)+\lambda^2\rho u=F+G+\Div\Vec{G},
     \end{align}
holds on $\R^n$,     
then
\begin{align}\label{ineq. Prop. Rn}
\|u\|_{L^q(Q)}\lesssim&  \delta\lambda^{\gamma(q)}
(\|u\|_{L^2}
+\lambda^{-1}\|\nabla u\|_{L^2}
+\|\Vec{G}\|_{L^2}
+(\eps\lambda)^{-1}\|G\|_{L^2}+\epsilon^{-1}\delta\lambda^{\gamma(q)-1}\|F\|_{L^{q'}}),
\end{align}
provided that the same estimate holds when $F=0$.
 \end{proposition}

\begin{remark}\label{remark reverse reductions}
$(i)$ With reference to \eqref{resolvent equation for phi u II}, we would   replace $F$ by $\rho\phi F$, $G$ by $\rho\phi G+a^{ij}\partial_i\phi\partial_j u$. This does not change the estimate (recall that $\epsilon\geq 1)$.

\noindent $(ii)$ The estimate with $F=0$ implies the spectral cluster bound~\eqref{spectral cluster estimate strong theorem}, see \cite[Cor. 5]{MR2280790}). Here, we will use the converse. Although it is rather obvious that the steps of the reduction can be reversed,
we will elaborate on this point in Subsection~\ref{subs. hyperboli eq.}. 

\noindent $(iii)$ Since in the following we will use the letter $a$ for symbols, to avoid confusion, we will again write $g^{ij}$ instead of $a^{ij}$. 
\end{remark}

\subsection{Symbol smoothing and factorization}\label{subsec. smoothing and factorization factorization}
We replace $g^{ij},\rho$ by smooth functions $g^{ij}_{\lambda},\rho_{\lambda}$, which are obtained by smoothly truncating to frequencies $|\xi|\ll \lambda$. 
\begin{lemma}\label{lemma symbol smoothing}
Let $g_{\lambda}^{ij}:=\varphi(D)g^{ij},\rho_{\lambda}(\xi):=\varphi(D)\rho$, where $\varphi$ is a bump function adapted to $|\xi|\ll \lambda$. Then
\begin{align*}
\|\rho-\rho_{\lambda}\|_{L^{\infty}}&\lesssim \lambda^{-1}\|\rho\|_{\mathrm{Lip}},\\
\|g^{ij}-g^{ij}_{\lambda}\|_{L^{\infty}}&\lesssim \lambda^{-1}\|g^{ij}\|_{\mathrm{Lip}}.
\end{align*}
\end{lemma}

\begin{proof}
We only prove this for $\rho$, the proof for $g^{ij}$ is identical.
Since $\varphi(0)=1$, we have
    \begin{align*}
        \rho(x)-\rho_{\lambda}(x)=\int_{\R^n}\varphi^{\vee}(x-y)(\rho(x)-\rho(y))\rd y,\quad x\in \R^n.
    \end{align*}
    Thus,
    \begin{align}
      |\rho(x)-\rho_{\lambda}(x)|\leq \|\rho\|_{\mathrm{Lip}(\R^n)}\int_{\R^n}|(x-y)\varphi^{\vee}(x-y)| \rd y. 
    \end{align}
By assumption, $\varphi(\xi)=\varphi_0(\xi/c\lambda)$ for some bump function $\varphi_0$ on the unit scale and some $c\ll 1$. Therefore, $\varphi^{\vee}(x)=(c\lambda)^n\varphi_0^{\vee}(c\lambda x)$. A change of variables thus yields the claimed bound.    
\end{proof}


By Lemma \ref{lemma symbol smoothing} we have for all $i=1,\ldots,n$,
\begin{align*}
    \|(g^{ij}-g^{ij}_{\lambda})\partial_j u\|_{L^2}\lesssim \lambda^{-1}\|\nabla u\|_{L^2}
\end{align*}
which means that the smoothing error $\partial_i(g^{ij}-g^{ij}_{\lambda})\partial_j u$
may be absorbed into the forcing term~$\Vec{G}$.
It thus suffices to prove \eqref{ineq. Prop. Rn} whenever
\begin{align*}
        \partial_i(g_{\lambda}^{ij}\partial_ju)+\lambda^2\rho_{\lambda} u=F+G+\Div\Vec{G}.
\end{align*}
To proceed, we split $u=u_{\ll \lambda}+u_{\lambda}+u_{\gg \lambda}$, where the terms are frequency-localized to $|\xi|\ll\lambda$, $|\xi|\approx\lambda$ and $|\xi|\gg \lambda$, respectively. The estimates for $u_{\ll\lambda}$ and $u_{\gg\lambda}$ follow from integration by parts and Sobolev embedding, similarly as in Proposition~ \ref{prop. elliptic}, and they do not require the spectral cluster bound (see also the proof of~\cite[Cor. 5]{MR2280790}). 

It remains to consider $u_{\lambda}:=\Gamma(D)u$, where $\Gamma$ is a bump function adapted to $|\xi|\approx \lambda$. This satisfies
\begin{align}\label{eq. for ulambda}
\partial_i(g_{\lambda}^{ij}\partial_ju_{\lambda})+\lambda^2\rho_{\lambda} u_{\lambda}=\Gamma(D)(F+G+\Div\Vec{G})+\partial_i([g_{\lambda}^{ij},\Gamma(D)]\partial_j u)+\lambda^2[\rho_{\lambda},\Gamma(D)]u.
\end{align}
Similarly as in Lemma \ref{lemma symbol smoothing} one easily proves
\begin{align}
\|[g_{\lambda}^{ij},\Gamma(D)]\partial_ju\|_{L^2}\lesssim\lambda^{-1}\|\nabla u\|_{L^2},\quad
\lambda^2\|[\rho_{\lambda},\Gamma(D)]u\|_{L^2}\lesssim \lambda\|u\|_{L^2}.
\end{align}
Hence the commutator errors can be absorbed into $\Vec{G}$ and $G$, respectively, without changing the right hand side of \eqref{ineq. Prop. Rn}. Moreover, due to the frequency truncation of $g^{ij}_{\lambda}$ and $\rho_{\lambda}$, all terms in \eqref{eq. for ulambda} are localized to frequencies $|\xi|\approx\lambda$. For instance, 
\begin{align*}
\supp(\widehat{\rho_{\lambda}u_{\lambda}})=\supp(\widehat{\rho_{\lambda}}\ast\widehat{u_{\lambda}})\subset \supp(\widehat{\rho_{\lambda}})+\supp(\widehat{u_{\lambda}}).
\end{align*}
Hence, choosing the implicit constant in the truncation $|\xi|\ll \lambda$ (see Lemma \ref{lemma symbol smoothing}) sufficiently small ensures that $\rho_{\lambda}u_{\lambda}$ is still supported in $|\xi|\approx\lambda$. Due to this frequency localization and since $\eps\geq 1$, 
\begin{align*}
    (\eps\lambda)^{-1}\|\Gamma(D)(G+\Div\Vec{G})+\partial_i([g_{\lambda}^{ij},\Gamma(D)]\partial_j u)+\lambda^2[\rho_{\lambda},\Gamma(D)]u\|_{L^2}\\
    \lesssim\|u\|_{L^2}+(\eps\lambda)^{-1}\|G\|_{L^2}+\|\Vec{G}\|_{L^2}.
\end{align*}
We may thus treat all the terms on the right of \eqref{eq. for ulambda} (except $F_{\lambda}:=\Gamma(D)F$) the same and call this collection of terms $G$ again. 
Note also that if \eqref{g close to Euclidean metric}, \eqref{a,rho close to Euclidean metric} hold, then by the triangle inequality we also have
\begin{align}\label{smoothed metric close to Euclidean}
\|\rho_{\lambda}-1\|_{\mathrm{Lip}}+ \|g_{\lambda}^{ij}-\delta^{ij}\|_{\mathrm{Lip}}\ll 1. 
\end{align}
For instance, by Lemma \ref{lemma symbol smoothing},
\begin{align*}
  \|\rho_{\lambda}-1\|_{L^{\infty}}\leq \|\rho-1\|_{L^{\infty}}+  \|\rho_{\lambda}-\rho\|_{L^{\infty}}\ll 1      
\end{align*}
for $\lambda\gg 1$ (which we may always assume), and $\|\nabla\rho_{\lambda}\|_{L^{\infty}}\lesssim \|\nabla\rho\|_{L^{\infty}}\ll 1$. By a partition of unity, we may actually assume that $\Gamma$ localizes to the conic region $|\xi'|\ll \xi_1\approx\lambda$, where $\xi=(\xi_1,\xi')\in\R\times\R^{n-1}$. 
It will thus suffice to prove the following version of Proposition \ref{Prop. Rn}.

\begin{proposition}\label{Prop. Rn freq. loc.}
     Assume that $g_{\lambda}^{ij}$, $\rho_{\lambda}$, globally defined on $\R^n$, are truncated to frequencies $|\xi|\ll \lambda$ and satisfy \eqref{smoothed metric close to Euclidean}.
     If
     \begin{align}\label{eq. Prop. Rn frq. loc.}
\partial_i(g_{\lambda}^{ij}\partial_ju)+\lambda^2\rho_{\lambda} u=F+G
\end{align}
holds as an equation on $\R^n$ and $u$, $F$, $G$ are Fourier supported in $|\xi'|\ll \xi_1\approx\lambda$, then     
\begin{align}\label{ineq. theorem Rn freq. loc.}
\|u\|_{L^q(Q)}\lesssim  \delta\lambda^{\gamma(q)}
(\|u\|_{L^2(\R^n)}+(\eps\lambda)^{-1}\|G\|_{L^2(\R^n)}
+\epsilon^{-1}\delta\lambda^{\gamma(q)-1}\|F\|_{L^{q'}(\R^n)}),
\end{align}
provided that the same estimate holds when $F=0$.
 \end{proposition}

To proceed, we relabel $x_1=t$ and use coordinates $(t,x)\in\R\times\R^{n-1}$ instead of $(x_1,x')$. Dual coordinates are denoted by $(\tau,\xi)$. 
We will use paradifferential calculus to reduce \eqref{ineq. theorem Rn freq. loc.} to a first order equation.
For this, we will work with the following symbol classes (cf.\ \cite{MR3282983}).
\begin{definition}[Symbol classes]
Let $a(t,x,\xi,\lambda)$ be a family of symbols in $(x,\xi)\in T^*\R^d$, depending on parameters $t\in\R$ and $\lambda\geq 1$. For $\kappa\in [0,1]$, we say $a(t,x,\xi,\lambda)\in S_{\lambda,\lambda^{\kappa}}$ if, for all multi-indices $\alpha,\beta$,
\begin{align}\label{symbols Slambda 1}
|\partial_{t,x}^{\alpha}\partial_{\xi}^{\beta}a(t,x,\xi,\lambda)|&\lesssim_{\alpha,\beta}\lambda^{-|\beta|+\kappa |\alpha|}.
\end{align}
We say $a(t,x,\xi,\lambda)\in C^kS_{\lambda,\lambda^{\kappa}}$ if the stronger bound 
\begin{align}\label{symbols Slambda 2}
|\partial_{t,x}^{\alpha}\partial_{\xi}^{\beta}a(t,x,\xi,\lambda)|&\lesssim_{\alpha,\beta}\lambda^{-|\beta|+\kappa(|\alpha|-k)_+}
\end{align}
holds, where $x_+=\max(x,0)$. 
\end{definition}
We will only need the case $k=1$, $\kappa=1$ and $d=n-1$. Symbols 
will initially be defined in the phase space region $B_{\lambda}=\{|(t,x)|\leq 1$, $|\xi|\ll \tau\approx\lambda\}$ and extended globally (consistently with the symbol class).
 We sometimes suppress the dependence on $\lambda$, i.e.\ we write $a(t,x,\xi)$ instead of $a(t,x,\xi,\lambda)$. 
 As remarked in \cite{MR3282983}, since $a\in S_{\lambda,\lambda}$ is simply a bounded family of zero order symbols in the Hörmander class $S^0_{0,0}$, rescaled by $(t,x,\xi)\mapsto (\lambda t,\lambda x,\xi/\lambda)$, the Kohn-Nirenberg quantization $a(t,x,D_x)$ and the Weyl 
quantization $a^w(t,x,D_x)$ are $L^2$ bounded (with uniform bounds in $t$ and $\lambda$) \cite{MR517939}. We will use the Weyl quantization here. 
The asymptotic laws for composition and adjoint for symbols $a,b\in C^1S_{\lambda,\lambda}$ hold to first order, viz.
\begin{equation}\label{first order calculus}
\begin{split}
    &a^w(t,x,D_{x})b^w(t,x,D_{x})=(ab)^w(t,x,D_{x})+\lambda^{-1}R_1,\\
&a^w(t,x,D_{x})^*=(\overline{a})^w(t,x,D_{x})+\lambda^{-1}R_2,
\end{split}
\end{equation}  
where the errors 
$R_1,R_2$ are quantizations of symbols in $S_{\lambda,\lambda}$ and hence bounded in $L^2(\R^{d})$.

Since $g$ is pointwise close to the flat metric (recall \eqref{g close to Euclidean metric}) and $\lambda\gg 1$, Lemma \ref{lemma symbol smoothing} shows that $g_{\lambda}^{11}(x)>0$ for $x\in Q$. We may then factorize 
\begin{align}\label{factorization}
-\langle (\tau,\xi),g_{\lambda}^{-1}(t,x)(\tau,\xi)\rangle+\lambda^2=g_{\lambda}^{11}(t,x)(\tau+\widetilde{a}(t,\xi',\lambda))(\tau-a(t,\xi',\lambda))
\end{align} 
on $B_{\lambda}$,
where $\widetilde{a},a>0$ are symbols in $\lambda C^1S_{\lambda,\lambda}$, $\langle\cdot,\cdot\rangle$ is the Euclidean inner product and $g_{\lambda}^{-1}$ is the matrix with entries $g_{\lambda}^{ij}$. 
We smoothly truncate the $(t,x)$ Fourier transform of $a$ to frequencies $|(\tau,\xi)|\ll\lambda$. The resulting symbol $a_{\lambda}$ is again in $\lambda C^1S_{\lambda,\lambda}$, and the difference $a-a_{\lambda}$ is in $S_{\lambda,\lambda}$ (the proof is similar to that of Lemma \ref{lemma symbol smoothing} and is omitted). We also extend $a_{\lambda}$ globally so that $a_{\lambda}(t,x,\xi)=\lambda$ for $|\xi|\gg\lambda$, while preserving the frequency localization in the $(t,x)$ variables. The same procedure is applied to $\widetilde{a}$.

Since the norm in Proposition \ref{Prop. Rn freq. loc.} is taken over the unit cube,
we can replace $u$ by $\chi^w u$, where $\chi$ is adapted to 
 $B_{\lambda}$ and $\chi^w=\chi^w(t,x,D_t,D_x)$. Indeed, if 
 $\chi(t,x,\tau,\xi)=1$ on the support of $\phi(t,x)\widehat{u}(\tau,\xi)$, then
$$\|\phi u-\chi^wu\|_{H^N(\R^n)}\lesssim\lambda^{-N}\|u\|_{H^{-N}(\R^n)}$$ for any $N>0$.
By Sobolev embedding, the error may thus be absorbed into the first term on the right of \eqref{ineq. theorem Rn freq. loc.}.
Moving forward, we may thus assume that $u$ is microlocalized, i.e.\ 
\begin{align}\label{u is microlocalized}
    u=\chi^wu+\mathcal{O}_{H^{-N}\to H^N}(\lambda^{-N})u \quad\forall N>0.
\end{align}
First order calculus \eqref{first order calculus} and the factorization \eqref{factorization} then yield, with $D=(D_t,D_x)$,
\begin{align}\label{factorization operators}
-\langle D,g_{\lambda}^{-1}(t,x)D\rangle u+\lambda^2u=-g_{\lambda}^{11}(t,x)(D_{t}+\widetilde{a}_{\lambda}^w(t,x,D_{x}))(D_{t}-a_{\lambda}^w(t,x,D_{x}))u+\lambda Ru
\end{align}
for some $R=\mathcal{O}_{L^2\to L^2}(1)$. Due to \eqref{u is microlocalized}, the factor $g^{11}_{\lambda}(t,x)(D_{t}+\widetilde{a}_{\lambda}(t,x,D_{x}))$ admits a parametrix $Q$ with symbol in $\lambda^{-1}S_{\lambda,\lambda}$.
Therefore, 
\begin{align*}
    -Q g_{\lambda}^{11}(t,x)(D_{t}+\widetilde{a}_{\lambda}^w(t,x,D_{x}))u=u+\lambda^{-1}R_1 u
\end{align*}
for some $R_1=\mathcal{O}_{L^2\to L^2}(1)$. Hence, if $u$ satisfies \eqref{eq. Prop. Rn frq. loc.}, then
\begin{align*}
    (D_{t}-a_{\lambda}^w(t,x,D_{x}))u=Q(F+G)-\lambda^{-1}R_1(D_{t}-a_{\lambda}^w(t,x,D_{x}))u-\lambda QRu.
\end{align*}
We write the right hand side as $F'+G'$, with $F'=QF$ and $\|G'\|_{L^2}\lesssim \|u\|_{L^2}+\lambda^{-1}\|G\|_{L^2}$. We will also
need that operators with $S_{\lambda,\lambda}$ symbols are $L^p$ bounded for $1<p<\infty$. Since we already know that they are $L^2$ bounded, this follows from Calderón--Zygmund theory since their Schwartz kernels satisfy standard estimates. In fact, after rescaling $(t,x,\xi)\mapsto (\lambda t,\lambda x,\xi/\lambda)$, it suffices to estimate the kernel $K$ of an operator corresponding to an $S_{0,0}$ symbol, localized to frequencies $|\xi|\lesssim 1$; this satisfies the trivial bound $|K(x,y)|+|\nabla_{x,y} K(x,y)|\lesssim 1$.
As a consequence, we have $\|F'\|_{L^{q'}}\lesssim \lambda^{-1}\|F\|_{L^{q'}}$. We work again with $F,G$ instead of $F',G'$. In other words, it suffices to prove the following proposition.
\begin{proposition}\label{Prop. first order 1}
Assume that $a_{\lambda}\in\lambda C^1S_{\lambda,\lambda}$ has $(t,x)$-frequencies $|(\tau,\xi)|\ll\lambda$ and $a_{\lambda}(t,x,\xi)=\lambda$ for $|\xi|\gg\lambda$. If
$u$ satisfies
\begin{align}\label{first order equation (revision)}
    (D_{t}-a_{\lambda}^w(t,x,D_{x}))u=F+G,
\end{align}
on $[0,1]\times\R^{n-1}$,
then we have 
\begin{align}\label{ineq. first order equation (revision) before Prop.}
    \|\chi^wu\|_{L^q([0,1]\times\R^{n-1})}\lesssim  \delta\lambda^{\gamma(q)}
(\|u\|_{L^2(\R^n)}+\eps^{-1}\|G\|_{L^2(\R^n)}
+\eps^{-1}\delta\lambda^{\gamma(q)}\|F\|_{L^{q'}(\R^n)}),
\end{align}
provided that the same estimate holds for $F=0$.
\end{proposition}

\subsection{Microlocalized quasimode estimate}
\label{subsct. Microlocalized quasimode estimate}
We will prove Proposition \ref{Prop. first order 1} by using the spectral cluster bound \eqref{spectral cluster estimate strong theorem} (for the second time). 
As explained in Remark \ref{remark reverse reductions} (ii), we 
first need to convert the spectral cluster bound to a suitable quasimode bound (coresponding to the case $F=0$ in Proposition \ref{Prop. Rn freq. loc.}). First, by the implication $(a)\implies (c)$ in Corollary \ref{corollary equivalence cluster and quasimode}, the spectral cluster bound \eqref{spectral cluster estimate strong theorem} implies the quasimode bound 
\begin{align}\label{spectral cluster reduction 1}
    \|u\|_{L^q(M)}\lesssim \delta\lambda^{\gamma(q)}(\|u\|_{L^2(M)}+(\eps\lambda)^{-1}\|(\Delta_g+\lambda^2)u\|_{L^2(M)})
\end{align}
for spectrally localized functions, i.e.\ $\Pi_{\leq 2\lambda}u=u$. We can remove the assumption of spectral localization since we have the stronger elliptic estimate \eqref{elliptic estimate Lq'L2} for $\Pi_{>2\lambda}u$. After extending $g$ globally (subject to \eqref{g close to Euclidean metric}) like before, we may consider this as a problem on $\R^n$.
Setting $\rho=|g|^{1/2}$, $a^{ij}=|g|^{1/2}g^{ij}$ and relabelling $a^{ij}$ to $g^{ij}$ again, we can replace $(\Delta_g+\lambda^2)u$ by $\partial_i(g^{ij}\partial_j u)+\lambda^2\rho u$, where we use that $\|\rho^{-1}\|_{L^{\infty}}\lesssim 1$, due to \eqref{a,rho close to Euclidean metric}. Therefore,
\begin{align}\label{spectral cluster reduction 2}
    \|u\|_{L^q(Q)}\lesssim \delta\lambda^{\gamma(q)}(\|u\|_{L^2(\R^n)}+(\eps\lambda)^{-1}\|\partial_i(g^{ij}\partial_j u)+\lambda^2\rho u\|_{L^2(\R^n)})
\end{align}
The key point is that this inequality is microlocalizable and stable under lower order perturbations (here we use again that $\eps\geq 1)$. After microlocalization to $B_{\lambda}$, we may thus again replace $g^{ij},\rho$ by the smoothed functions $g_{\lambda}^{ij},\rho_{\lambda}$ and use the first order calculus, in particular the factorization \eqref{factorization operators},
to obtain the microlocalized quasimode estimate
\begin{align}\label{microlocalized quasimode} 
\|\chi^wu\|_{L^q_{t,x}}
\lesssim \delta\lambda^{\gamma(q)}(\|u\|_{L^2_{t,x}}+\eps^{-1}\|
    (D_{t}-a_{\lambda}^w(t,x,D_x))u\|_{L^2_{t,x}})
\end{align}
with norms over $(t,x)\in[0,1]\times \R^{n-1}$ and for all $u\in L^2_{t,x}$.

\subsection{From homogeneous to inhomogeneous estimates}\label{subs. hyperboli eq.}
We will use the Christ--Kiselev lemma to upgrade \eqref{microlocalized quasimode} to the stronger estimate \eqref{ineq. first order equation (revision) before Prop.}. In view of the discussion after \eqref{ineq. first order equation (revision) before Prop.}, it remains to prove the following lemma.


\begin{lemma}\label{lemma first order 2}
Assume that the microlocalized quasimode estimate \eqref{microlocalized quasimode} holds for every bump function $\chi$ supported in $B_{\lambda}=\{|(t,x)|\leq 1$, $|\xi|\ll \tau\approx\lambda\}$. 
Let $[t_0,t_0+T]\subset[0,1]$ and assume that $\chi,\widetilde{\chi}$  are supported in $B_{\lambda}$. Assume that 
     $u$ satisfies
\begin{align}\label{first order hyperbolic equation}
(D_{t}-a_{\lambda}^w(t,x,D_{x}))u=\widetilde{\chi}^wF+G,\quad u(t_0)=u_0,
\end{align}
on $[t_0,t_0+T]\times \R^{n-1}$. Then we have
\begin{align}\label{ineq. first order equation (revision)}
\|\chi^wu\|_{L^q_{t,x}}\lesssim  \delta\lambda^{\gamma(q)}T^{1/2}
(\|u_0\|_{L^2_{x}}+T^{1/2}\|G\|_{L^2_{t,x}}
+\delta\lambda^{\gamma(q)}T^{1/2}\|F\|_{L^{q'}_{t,x}}),
\end{align}
with norms over $[t_0,t_0+T]\times \R^{n-1}$. Moreover, we have the energy estimate
\begin{align}\label{energy estimate (revision)}
\|u\|_{L^{\infty}_tL^2_x}\lesssim  \|u_0\|_{L^2_{x}}+T^{1/2}\|G\|_{L^2_{t,x}}
+\delta\lambda^{\gamma(q)}T^{1/2}\|F\|_{L^{q'}_{t,x}}.
\end{align}
 \end{lemma}

\begin{proof}
Without loss of generality we may assume $t_0=0$ (the symbol bounds on $a_{\lambda}$ are uniform in $t$ over a compact interval).
  Let $S(t,s):L^2(\R^{n-1})\to L^2(\R^{n-1})$ be the unitary propagator associated to the bounded self-adjoint operator $a_{\lambda}^w$; it satisfies
\begin{align*}
    (D_{t}-a_{\lambda}^w(t,x,D_{x}))S(t,s)u_0=0,\quad S(s,s)u_0=u_0
\end{align*}
for $t,s\in[0,T]$ and $u_0\in L^2(\R^{n-1})$. Observe the group property $S(t,r)S(r,s)=S(t,s)$ and $S(t,s)^*=S(t,s)^{-1}=S(s,t)$.
Set $S(t)=S(t,0)$.
Applying \eqref{microlocalized quasimode} to $S(t)u_0(x)$ we get
\begin{align}\label{S(t,s)}
    \|\chi^wS(t)u_0\|_{L_{t,x}^q}\lesssim \delta\lambda^{\gamma(q)}\|S(t)u_0\|_{L_{t,x}^2}\leq \delta\lambda^{\gamma(q)} T^{1/2}\|S(t)u_0\|_{L_{t}^{\infty}L_{x}^{2}}
    \leq\delta\lambda^{\gamma(q)} T^{1/2}\|u_0\|_{L_{x}^{2}},
\end{align}
where we used the unitarity of $S(t)$ and H\"older in the last estimate (recalling that $t\in[0,T]$).
By duality, this implies that
\begin{align}\label{dual S(t,s)}
    \|\int_0^{t}S(s)^*\chi^wf(s)\rd s\|_{L^{\infty}_tL^2_x}\lesssim \delta\lambda^{\gamma(q)} T^{1/2}\|f\|_{L^{q'}_{t,x}}.
\end{align}
Indeed, for $u_0\in L^2_x$, $f\in L^{q'}_{t,x}$, we have
\begin{align*}
    \sup_{t\in[0,T]}|\langle\int_0^{t}S(t)^*\chi^wf(t)\rd t,u_0\rangle_{L^2_x}|\leq \|f\|_{L^{q'}_{t,x}}\|\chi^wS(t)u_0\|_{L_{t,x}^q}\lesssim \delta\lambda^{\gamma(q)} T^{1/2}\|f\|_{L^{q'}_{t,x}}\|u_0\|_{L_{x}^{2}}.
\end{align*}
Of course, this also holds when $\chi^w$ is replaced by $\widetilde{\chi}^w$. Hence, if 
$f,g\in L^{q'}_{t,x}$, then 
\begin{align*}
&|\langle\int_0^{T}\chi^wS(t)S(s)^*\widetilde{\chi}^wf(s)\rd s,g\rangle_{L^2_{t,x}}|\\
&=|\langle\int_0^{T}S(s)^*\widetilde{\chi}^wf(s)\rd s,\int_0^{T}S(t)^*\chi^wg(t)\rd t\rangle_{L^2_{x}}|\\
&\leq \|\int_0^{T}S(s)^*\tilde{\chi}^wf(s)\rd s\|_{L^2_{x}} \|\int_0^{T}S(t)^*\chi^wg(t)\rd t\|_{L^2_{x}}\\
&\lesssim \delta^2\lambda^{2\gamma(q)}T\|f\|_{L^{q'}_{t,x}}\|g\|_{L^{q'}_{t,x}}.
\end{align*}
By the group property, this implies that
\begin{align}\label{TTstar S(t,s)}
   \|\int_0^{T}\chi^wS(t,s)\tilde{\chi}^wF(s)\rd s\|_{L^q_{t,x}}\lesssim \delta^2\lambda^{2\gamma(q)}T\|F\|_{L^{q'}_{t,x}}.
\end{align}
Thus, an application of the Christ--Kiselev lemma \cite{MR1809116}, as formulated e.g.\ in  \cite{MR1789924,MR2233925}, yields 
\begin{align}\label{Christ--Kiselev}
    \|\int_0^{t}\chi^wS(t,s)\tilde{\chi}^wF(s)\rd s\|_{L^q_{t,x}}\lesssim \delta^2\lambda^{2\gamma(q)}T\|F\|_{L^{q'}_{t,x}}.
\end{align}
By unitarity of $S(t,s)$, \eqref{S(t,s)} also implies
\begin{align*}
         \sup_{s\in [0,T]}\|\chi^wS(t,s)u_0\|_{L_{t,x}^q}\lesssim \delta\lambda^{\gamma(q)} T^{1/2}\|u_0\|_{L_{x}^{2}}.
\end{align*}
Minkowski's and H\"older's inequality then yield
\begin{align}\label{Mink}
    \|\int_0^{T}\mathbf{1}_{s<t} \chi^wS(t,s)G(s)\rd s\|_{L^q_{t,x}}\leq     
    \int_0^{T}\|\chi^wS(t,s)G(s)\|_{L^q_{t,x}}\rd s\lesssim 
    \delta\lambda^{\gamma(q)}T\|G\|_{L^2_{t,x}}.
\end{align}  
Assume now that $u$ satisfies \eqref{first order hyperbolic equation}. By Duhamel's formula,
\begin{align}\label{Duhamel}
u(t)=S(t)u_0+\I\int_0^t S(t,s)G(s)\rd s+\I\int_0^t S(t,s)\widetilde{\chi}^wF(s)\rd s.
\end{align}
Hence \eqref{ineq. first order equation (revision)} follows from \eqref{Mink}, \eqref{Christ--Kiselev} and \eqref{S(t,s)}.
Further, \eqref{Duhamel} and unitarity of $S(t,s)$ implies
\begin{align}
    \|u\|_{L^{\infty}_tL^2_x}&\leq \|u_0\|_{L^2_x}+\int_0^T\|G(s)\|_{L^2_x}\rd s+\|\int_0^tS(t,s)\widetilde{\chi}^wF(s)\rd s\|_{L^{\infty}_tL^2_x}\\
   & \lesssim \|u_0\|_{L^2_x}+T^{1/2}\|G\|_{L^2_{t,x}}+\delta\lambda^{\gamma(q)} T^{1/2}\|F\|_{L^{q'}_{t,x}},
\end{align}
where we used \eqref{dual S(t,s)} for the third term. In addition, we used that, by the group property, $S(t,s)=S(s,t)^*=(S(s)S(t)^*)^*=S(t)S(s)^*$. This proves the energy estimate \eqref{energy estimate (revision)}.
\end{proof}

\begin{remark}
The proof is similar to that of \cite[Prop. 4.8]{MR2094851}. The first difference is that we omitted the curvature condition (A3) in \cite{MR2094851} and replaced it by the (black box) assumption that the spectral cluster bound (or the equivalent quasimode estimate \eqref{microlocalized quasimode}) holds. In contrast, the real principal type condition (A2) in \cite{MR2094851} is crucial for the reduction to the first order equation \eqref{first order equation (revision)}. 
In the context of Theorem~\ref{theorem spectral cluster to resolvent intro}, the condition (A3) is of course also satisfied:
The fibers of the cosphere bundle $|\xi|_g=1$ have everywhere nonzero Gaussian curvature. But this information is now encoded in the spectral cluster bound. The second difference is that \cite[Proposition 4.8]{MR2094851} applies only to symbols $a_{\lambda^{1/2}}$ in the smaller class $C^2S_{\lambda,\lambda^{1/2}}$, at which one would arrive, after similar reductions as in the proof of Theorem \ref{theorem spectral cluster to resolvent intro}, if one had started with a $C^{1,1}$ metric. Lipschitz functions are well approximated by $C^{1,1}$ functions on the $\lambda^{-1/3}$ scale. On this scale the wave packet parametrix constructions work,
and one can prove the spectral cluster (homogeneous Strichartz) and resolvent (inhomogeneous Strichartz) bounds concurrently. These follow from dispersive ($L^1\to L^{\infty}$) estimates on the regularized flow corresponding to $a_{\lambda^{1/2}}$. For such estimates, the causal factor $\mathbf{1}_{t\geq s}$ is innocuous. On the other hand, the flow of $a_{\lambda}$ in Lemma~\ref{lemma first order 2} is not regular enough for the dispersive estimates to hold. The third difference is that we use a microlocalization in the full phase space $T^*\R_{t,x}^n$, whereas Koch and Tataru \cite{MR2094851} only assume microlocalization in $T^*\R^{n-1}_x$ (hence their estimate is stronger). We could achieve the same if we assume that \eqref{microlocalized quasimode} holds under this weaker assumption. However, such an estimate does in general not follow from the spectral cluster estimate \eqref{spectral cluster estimate strong theorem}.
\end{remark}

\subsection{Proof of Proposition \ref{Prop. first order 1}}\label{subs. proof eps>1}\label{subsec. proof eps>1}
We show that Lemma \ref{lemma first order 2} implies Proposition \ref{Prop. first order 1}, which will finish the proof of Theorem \ref{theorem spectral cluster to resolvent stronger version} in the case $\eps\geq 1$. 

Let $\widetilde{\chi}$ be another bump function adapted to 
$B_{\lambda}$ and such that $\widetilde{\chi}=1$ on $\supp\chi$.
First order calculus allows us to replace $F$ in \eqref{first order equation (revision)} by $\widetilde{\chi}^wF$ (the error can be absorbed into $G$ without changing the estimate).
We also note that \eqref{first order equation (revision)} is a local equation in $t$ since $a_{\lambda}^w(t,x,D_{x})$ commutes with multiplication by functions of $t$ only; thus,
\begin{align}\label{localization with phi(t)}
        (D_{t}-a_{\lambda}^w(t,x,D_{x}))(\phi(t) u)=\phi(t)(\widetilde{\chi}^wF+G)-\I\phi'(t)u
\end{align}
on $\R^{n}$.
 Let $1=\sum_{j=1}^K\phi_j(t)$, $t\in [0,1]$, (with $K\approx\eps$) be a smooth partition of unity such that $\phi_j$ is compactly supported in an interval $I_j$ of length $\eps^{-1}$ and $\|\phi_j'\|_{L^{\infty}}\lesssim \eps$.
 Let $\psi_j$ be another bump function supported in $I_j$ such that $\psi_j(t)=1$ on the support of $\phi_j$.
If $\phi$ in \eqref{localization with phi(t)} is one the $\psi_j$'s, the right hand side can be decomposed as $\chi_1^wF_j+G_j$, with $F_j=\psi_j(t)\widetilde{\chi}^wF$ and $G_j=\psi_j(t)G-\I\psi_j'(t)u+(1-\chi_1^w)\psi_j(t)\widetilde{\chi}^wF$, where $\chi_1$ is supported in $B_{\lambda}$ and equals $1$ on the support of $\widetilde{\chi}$. In particular,  
 \begin{align}
     \|G_j\|_{L^2(\R^n)}&\lesssim \|G\|_{L^2(I_j\times\R^{n-1})}     +\eps\|u\|_{L^2(I_j\times\R^{n-1})}+\lambda^{-N}\|F\|_{H^{-N}(\R^n)},\label{sum Gj}\\
\|F_j\|_{L^{q'}(\R^n)}&\lesssim \|F\|_{L^{q'}(I_j\times\R^{n-1})}\label{sum Fj},
 \end{align}
 where we used standard pseudodifferential calculus for the first estimate. In the following, $N>0$ is arbitrary and my change from line to line (the constant of course depends on $N$, but we will not indicate this). Lemma \ref{lemma first order 2} (with $T=\eps^{-1}$) 
 implies
 \begin{align*}
\|\chi^w(\psi_ju)\|_{L^q(I_j\times\R^{n-1})}&\lesssim  \delta\lambda^{\gamma(q)}\eps^{-1}
(\|G_j\|_{L^2(\R^{n})}
+\delta\lambda^{\gamma(q)}\|F_j\|_{L^{q'}(\R^{n})}),
 \end{align*}
 for all $1\leq j\leq K$. Since $\phi_j(t)(1-\psi_j(t))=0$, it follows again by standard pseudodifferential calculus that
 \begin{align}
     \|\phi_j\chi^wu\|_{L^q(I_j\times\R^{n-1})}&\lesssim  \delta\lambda^{\gamma(q)}\eps^{-1}
(\|G_j\|_{L^2(\R^n)}
+\delta\lambda^{\gamma(q)}\|F_j\|_{L^{q'}(\R^n)})+\lambda^{-N}\|u\|_{H^{-N}(\R^n)}.
 \end{align}
 Raising this to the power $q$ and summing over $j$ yields 
 \begin{align}
    \|\chi^wu\|_{L^q(I\times\R^{n-1})}^q&\lesssim  (\delta\lambda^{\gamma(q)})^q
\big(\eps^{-q}\sum_{j=1}^K\|G\|_{L^2(I_j\times\R^{n-1})}^q+\|u\|_{L^2(I_j\times\R^{n-1})}^q\big)\\
&+(\delta^2\lambda^{2\gamma(q)}\eps^{-1})^q\sum_{j=1}^K\|F\|_{L^{q'}(I_j\times\R^{n-1})}^q+\lambda^{-N}\|u\|_{H^{-N}(\R^n)}^q+\lambda^{-N}\|F\|_{H^{-N}(\R^n)}^q,
 \end{align}
where we used \eqref{sum Gj} and \eqref{sum Fj}. Since $q'\leq 2\leq q$ we have $\ell^q\supset\ell^2\supset \ell^{q'}$ and hence
\begin{align*}
  \|\chi^wu\|_{L^q(I\times\R^{n-1})}^q&\lesssim  (\delta\lambda^{\gamma(q)})^q
\big(\eps^{-2}\sum_{j=1}^K\|G\|_{L^2(I_j\times\R^{n-1})}^2+\|u\|_{L^2(I_j\times\R^{n-1})}^2\big)^{q/2}\\
&+(\delta^2\lambda^{2\gamma(q)}\eps^{-1})^q\big(\sum_{j=1}^K\|F\|_{L^{q'}(I_j\times\R^{n-1})}^{q'}\big)^{q/q'}+\lambda^{-N}(\|u\|_{H^{-N}(\R^n)}^q+\|F\|_{H^{-N}(\R^n)}^q)\\
&\lesssim (\delta\lambda^{\gamma(q)})^q(\eps^{-q}\|G\|_{L^2(\R^{n})}^q+\|u\|_{L^2(\R^{n})}^q)+(\delta^2\lambda^{2\gamma(q)}\eps^{-1})^q \|F\|_{L^{q'}(\R^n)}^q,
\end{align*}
where we used (dual) Sobolev embedding $L^{q'}\subset H^{-N}$ and $\delta\eps^{-1}\geq \lambda^{-1/2}$ to absorb the $\lambda^{-N}$ terms. This concludes the proof of Proposition \ref{Prop. first order 1}.

\subsection{The case $\epsilon< 1$}\label{subs. eps<1}

We now turn to the case $\epsilon<1$.
In view of the elliptic estimates in Proposition \ref{prop. elliptic} (case (b) applies to arbitrary $\eps$) it is sufficient to prove 
\begin{align*}
\|(\Delta_g+(\lambda+\I \epsilon)^2)^{-1}\chi(\sqrt{-\Delta_g}/\lambda)\|_{L^{q'}\to L^{q}}\lesssim \epsilon^{-1}\delta^2\lambda^{2\gamma(q)-1}.
\end{align*}
where $\chi(\tau)$ is a bump function adapted to $|\tau|\leq 2$ (here $\tau$ is unrelated to the dual variable of $t$ in the previous subsections). We partially follow the strategy of \cite{BourgainShaoSoggeEtAl2015} and represent the resolvent operator as a cosine transform,
\begin{align*}
\mathfrak{S}f:=(\Delta_g+(\lambda+\I\epsilon)^2)^{-1}f=\frac{1}{\I(\lambda+\I\epsilon)}\int_0^{\infty}\e^{\I \ell\lambda-\ell\epsilon}\cos(\ell\sqrt{-\Delta_g})f\rd \ell.
\end{align*} 
We will also need the localized version
\begin{align*}
\mathfrak{S}_{\rm loc} f=\frac{1}{\I(\lambda+\I\epsilon)}\int_0^{\infty}\rho(\ell)\e^{\I \ell\lambda-\ell\epsilon}\cos(\ell\sqrt{-\Delta_g})f\rd \ell,
\end{align*}
where $\rho$ is a bump function adapted to $[-\delta_0,\delta_0]$, with $\delta_0\ll 1$ smaller than half the injectivity radius of $(M,g)$. 
We will show that
\begin{align}\label{local bound}
    \|\mathfrak{S}_{\rm loc}\chi(\sqrt{-\Delta_g}/\lambda)\|_{L^{q'}\to L^q}\lesssim \lambda^{2\gamma(q)-1},
\end{align}
This is better than the right hand side of \eqref{resolvent estimate strong theorem} since $\delta^2\epsilon^{-1}\geq 1$. for the proof we write 
\begin{align*}
    \mathfrak{S}_{\rm loc}f=(\lambda+\I\epsilon)^{-1}\sum_{\pm}\psi(\sqrt{-\Delta_g}\pm \lambda)f,
\end{align*}
where $\psi$ is the Fourier transform of $c\rho(\ell)\e^{-\epsilon\ell}H(\ell)$; here $H$ is the Heaviside function and $c$ is a constant that will change from line to line. 
Then
\begin{align*}
 |\partial^k\psi(\tau)|=|\mathcal{F}(c\ell^k\rho(\ell)\e^{-\epsilon\ell}H(\ell))|\leq c\int_{0}^{\infty}\ell^k\rho(\ell)\rd \ell,   
\end{align*}
which means that $\psi$ is smooth on the unit scale (i.e.\ bounds independent of $\eps$). 
Moreover,
since the Fourier transform of $\e^{-\epsilon\ell}H(\ell)$ equals $c(\tau+\I\epsilon)^{-1}$, the convolution theorem yields that
\begin{align*}
    \psi(\tau)=c\int_{-\infty}^{\infty}\frac{\widehat{\rho}(\sigma)}{(\tau-\sigma)+\I\epsilon}\rd\sigma,
\end{align*}
from which we will conclude that $\psi(\tau)$ decays like $|\tau|^{-1}$.
Indeed, by the (easy part of) the Paley--Wiener theorem, $\widehat{\rho}$ is an entire function with $|\widehat{\rho}(\sigma)|\lesssim\langle \sigma\rangle^{-N}\e^{\delta_0|\im \sigma|}$. Consider the contour $\Gamma_R$ with vertices $\tau-R$, $\tau+R$, $\tau+R+2\I$ and $\tau-R+2\I$, transversed counterclockwise and with $R>|\tau|\gg 1$. The residue theorem yields
\begin{align*}
    \int_{\Gamma_R}\frac{\widehat{\rho}(\sigma)}{(\tau-\sigma)+\I\epsilon}\rd\sigma=\widehat{\rho}(\tau+\I\eps)=\mathcal{O}(\tau^{-N}).
\end{align*}
The integrals over the vertical sides are $\mathcal{O}(R^{-1})$, while the integral over the (top) horizontal side yields the leading contribution $\mathcal{O}(|\tau|^{-1})$. This shows that $|\psi(\tau)|\lesssim |\tau|^{-1}$. More generally, since
\begin{align*}
   \partial^k \psi(\tau)=c\int_{-\infty}^{\infty}\frac{\widehat{\rho}(\sigma)}{[(\tau-\sigma)+\I\epsilon]^{k+1}}\rd\sigma,
\end{align*}
the same argument (note that the residue vanishes for $k\geq 1$) yields the symbol type bounds $|\partial^k\psi(\tau)|\leq C_k\langle \tau\rangle^{-1-k}$. 
Hence, if we set 
\begin{align*}
    \eta(\tau)=(\lambda+\I\epsilon)^{-1}\sum_{\pm}\psi(\tau\pm \lambda)(\tau^2-(\lambda+\I)^2)\chi(\tau/\lambda), 
\end{align*}
we have that $|\partial^k \eta(\tau)|\leq C_k\langle \tau\rangle^{-k}$. For instance,
\begin{align*}
    |\eta(\tau)|\leq \lambda^{-1}\sum_{\pm}\langle \tau\pm\lambda\rangle\chi(\tau/\lambda)\lesssim 1
\end{align*}
since $\tau\lesssim\lambda$ on the support of $\chi(\tau/\lambda)$.
By standard multiplier theorems (e.g.\ \cite[Corollary 4.3.2]{MR3645429}), we have $\|\eta(\sqrt{-\Delta_g})\|_{L^p\to L^p}\leq C_p$ for $1<p<\infty$. Thus, by the fist part of the proof (i.e.\ the resolvent estimate for $\epsilon=1$), we have
\begin{align*}
    \|\mathfrak{S}_{\rm loc}\chi(\sqrt{-\Delta_g}/\lambda)\|_{L^{q'}\to L^q}\leq 
    \|\eta(\sqrt{-\Delta_g}/\lambda)\|_{L^{q}\to L^q}\|(\Delta_g+(\lambda+\I)^2)^{-1}\|_{L^{q'}\to L^q}\lesssim \lambda^{2\gamma(q)-1}.
\end{align*}
The global part $\mathfrak{S}-\mathfrak{S}_{\rm loc}$ can be estimated exactly as in \cite{BourgainShaoSoggeEtAl2015}. Since the argument there is only given for $q=2n/(n-2)$, $n\geq 3$, we provide some details and simplifications for the case considered here. By definition, $\mathfrak{S}-\mathfrak{S}_{\rm loc}=\lambda^{-1}m(\sqrt{-\Delta_g})$, where $m$ is the multiplier
\begin{align*}
m(\tau)=\frac{\lambda}{\I(\lambda+\I\epsilon)}\int_0^{\infty}(1-\rho(\ell))\e^{\I \ell\lambda-\ell\epsilon}\cos(\ell\tau)\rd \ell.
\end{align*}
The main observation is that, since $1-\rho(\ell)$ vanishes near the origin, there are no boundary terms and we can integrate by parts as many times as we please. Using Euler's formula for the cosine function, we deduce that 
\begin{align*}
|m(\tau)|\lesssim (1+|\tau-\lambda|)^{-N}+\epsilon^{-1}(1+\epsilon^{-1}|\tau-\lambda|)^{-N},
\end{align*}
 see \cite[Lemma 2.6]{BourgainShaoSoggeEtAl2015}.
 We now use Lemma \ref{lemma spectral multiplier bound} with $m_i(\tau)=\sqrt{m(\tau)\chi(\tau/\lambda)}$, $i=1,2$, and
 \begin{align*}
 M_j^2&\lesssim\sum_{k\lesssim\lambda}(1+|k-\lambda|)^{-N}\|\Pi_{[k,k+1]}\|^2_{L^{q'}\to L^2}
 +\epsilon^{-1}(1+\epsilon^{-1}|k-\epsilon^{-1}\lambda|)^{-N}\|\Pi_{[\epsilon k,\epsilon(k+1)]}\|^2_{L^{q'}\to L^2}\\
 &\lesssim \lambda^{2\gamma(q)}+\epsilon^{-1}\delta^2\lambda^{2\gamma(q)}\lesssim \epsilon^{-1}\delta^2\lambda^{2\gamma(q)}.
 \end{align*}
 This yields
 \begin{align*}
     \|(\mathfrak{S}-\mathfrak{S}_{\rm loc})\chi(\sqrt{-\Delta_g}/\lambda)\|_{L^{q'}\to L^q}=\lambda^{-1}\|m(\sqrt{-\Delta_g})\|_{L^{q'}\to L^q}\lesssim \epsilon^{-1}\delta^2\lambda^{2\gamma(q)-1},
 \end{align*}
which completes the proof in the case $\epsilon<1$.

\section{Stability under perturbations}
We will show that the resolvent estimates are stable under perturbations $V=V(\lambda)$ with the mapping properties
\begin{align}\label{mapping properties W}
V:W_{\lambda}^{\frac{1}{2},2}\cap W_{\lambda}^{s(q),q}\to W_{\lambda}^{-\frac{1}{2},2}+W_{\lambda}^{-s(q),q'},
\end{align}
Here, $1-s(q)=n(1/2-1/q)$ and $W_{\lambda}^{s,p}$ are Sobolev spaces flattened at frequency $\lambda$,
\begin{align*}
\|u\|_{W_{\lambda}^{s,p}}=\|(\lambda^2-\Delta_g)^{s/2}u\|_{L^p}.
\end{align*}
We assume that we have perfect estimates for the exponent $q$ under consideration, i.e.
\begin{align}\label{z=lambda=i perfect}
\|(\Delta_g+(\lambda+\I)^2)^{-1}\|_{L^{q'}\to L^{q}}\lesssim \lambda^{2\sigma(q)-1},
\end{align}
where we recall that $\sigma(q)=n(1/2-1/q)-1/2$ for $q\geq q_n$. This makes the the estimates compatible with Sobolev embedding. Note that the spaces in \eqref{mapping properties W} are isomorphic as $\lambda$ varies, but we ask for uniform bounds in $\lambda$. More precisely, we will assume that
\begin{align}\label{def. M(lambda)}
M(\Lambda):=\sup_{\lambda\geq \Lambda}\|V(\lambda)\|_{X(\lambda)\to X'(\lambda)}<\infty
\end{align}
for some (and hence all) $\Lambda\geq 1$. Here we introduced the spaces
\begin{align*}
X(\lambda)=W_{\lambda}^{\frac{1}{2},2}\cap W_{\lambda}^{s(q),q},\quad
X'(\lambda)=W_{\lambda}^{-\frac{1}{2},2}+W_{\lambda}^{-s(q),q'}.
\end{align*}

To prove the stability we first add in the $L^2$ estimates in the $L^{q'}\to L^q$ resolvent bound.

\begin{lemma}\label{lemma add in L2 bounds}
Assume that \eqref{z=lambda=i perfect} holds. Then we have 
\begin{align*}
\|(\Delta_g+(\lambda+\I)^2)^{-1}u\|_{\lambda^{-1/2}L^2\cap \lambda^{\sigma(q)-1/2}L^{q}}\lesssim \|u\|_{\lambda^{1/2}L^2+\lambda^{-\sigma(q)+1/2}L^{q'}}.
\end{align*}
\end{lemma}

\begin{proof}
The $L^{q'}\to L^q$ bound holds by assumption. The $L^2\to L^q$ bound is equivalent to the spectral cluster bound (implication $(a)\iff (b)$ in Corollary \ref{corollary equivalence cluster and quasimode}) and hence follows from the $L^{q'}\to L^q$ bound (Proposition \ref{proposition abstract bounds}). The $L^{q}\to L^2$ bound then follows by duality. The $L^2\to L^2$ bound follows from the spectral theorem.
\end{proof}

Next we add in the elliptic $W_{\lambda}^{-1,2}\to W_{\lambda}^{1,2}$ estimates, similar to \cite[Theorem 7]{MR2252331}, to obtain an even stronger version. 

\begin{lemma}\label{lemma add in elliptic bounds}
Assume that \eqref{z=lambda=i perfect} holds. Then there exists $C_0>0$ such that
\begin{align}\label{def. C0}
\|(\Delta_g+(\lambda+\I)^2)^{-1}\|_{X'(\lambda)\to X(\lambda)}\leq C_0.
\end{align}
\end{lemma}

\begin{proof}
This follows from the previous lemma together with the Sobolev embeddings
\begin{align}\label{Sobolev embeddings}
W_{\lambda}^{1,2}\subset W_{\lambda}^{s(q),q},\quad W_{\lambda}^{-s(q),q'}\subset W_{\lambda}^{-1,2},
\end{align}
and the observation that for functions localized to frequency $\lesssim\lambda$, we have
\begin{align*}
\|u\|_{W_{\lambda}^{s(q),q}}\lesssim \|u\|_{\lambda^{-s(q)}L^{q}},\quad \|u\|_{\lambda^{s(q)}L^{q'}}\lesssim \|u\|_{W_{\lambda}^{-s(q),q'}}.
\end{align*}
Note again that $s(q)=-\sigma(q)+1/2$.
\end{proof}

\begin{remark}
The $L^2$ based bounds allow one to handle gradient perturbations or perturbations of the metric, similarly to \cite{MR2252331}, \cite{MR2219246}, \cite{MR3615545}. 
\end{remark}

We will assume that for some $c\in(0,1)$,
\begin{align}\label{def. Lambda(eps)}
\Lambda_0=\inf\{\Lambda\geq 1:\,M(\Lambda)\leq c\, C_0^{-1}\}<\infty.
\end{align}

\begin{proposition}\label{proposition stability under perturbations effective}
Assume that \eqref{def. C0} holds and that $V=V(\lambda)$ is a family of perturbations satisfying \eqref{def. M(lambda)}, \eqref{def. Lambda(eps)}. Then
\begin{align}\label{effective bound V(lambda)}
\|(\Delta_g+V+(\lambda+\I)^2)^{-1}\|_{X'(\lambda)\to X(\lambda)}\leq C_0(1-c)^{-1}
\end{align}
for all $\lambda\geq \Lambda_0$. 
\end{proposition}

\begin{proof}
This is an immediate consequence of Lemma \ref{lemma add in elliptic bounds}, the second resolvent formula
\begin{align*}
(\Delta_g+V+(\lambda+\I)^2)^{-1}=(\Delta_g+(\lambda+\I)^2)^{-1}(\mathbf{1}+V(\Delta_g+(\lambda+\I)^2)^{-1})
\end{align*}
and a Neumann series argument.
\end{proof}

The estimate \eqref{effective bound V(lambda)} is effective in the sense that it is valid for a whole family of perturbations satisfying appropriate bounds. We are now turning our attention to non-effective bounds, i.e.\ we fix a perturbation $V$ and allow the constants in the estimates depend on $V$ itself. We assume that \eqref{mapping properties W} holds and take norms for fixed $\lambda$, say $\lambda=1$. For simplicity we only consider zero order perturbations (potentials). That is, we assume that $V\in L^p(M)$ with $1/p=1/q'-1/q$. Since $M$ is compact, we have the inclusion $L^{p_2}(M)\subset L^{p_1}(M)$ for $p_2\geq p_1$. 

\begin{proof}[Proof of Theorem \ref{theorem spectral cluster to resolvent intro}]
If $n=2$ or $n=3$ and $q<2^*$, then by Lemma \ref{lemma add in L2 bounds} we have 
\begin{align*}
M(\Lambda)\leq \Lambda^{2\sigma(q)-1}\|V\|_{L^p}.
\end{align*}
If we fix $c=1/2$ in the definition of $\Lambda_0$, as we do now, this implies that
\begin{align*}
\Lambda_0\leq (2C_0\|V\|_{L^p})^{\frac{1}{1-2\sigma(q)}}.
\end{align*}
Note that $1-2\sigma(q)>0$ here. The claim then follows from Proposition \ref{proposition stability under perturbations effective}.

Now consider the case $n\geq 3$ and $q=2^*$, so that $2\sigma(q)-1=0$. For $\alpha_0>0$ we set
\begin{align*}
V_{>\alpha_0}=V\mathbf{1}_{\{x\in M:\, |V(x)|>\alpha_0\}},\quad V_{\leq\alpha_0}=V-V_{>\alpha_0}.
\end{align*}
Since $V\in L^{n/2}$ we may choose $\alpha_0$ so large that $\|V_{>\alpha_0}\|_{L^{n/2}}\leq C_0^{-1}/4$. Since $\|V_{\leq\alpha_0}\|_{L^{\infty}}\leq \alpha_0$ we have
\begin{align*}
\|V\|_{X(\lambda)\to X'(\lambda)}\leq \|V_{>\alpha_0}\|_{L^{q'}\to L^q}+\lambda^{-1}\|V_{\leq\alpha_0}\|_{L^{2}\to L^2}\leq \tfrac{1}{4} C_0^{-1}+\lambda^{-1}\alpha_0,
\end{align*}
which yields
\begin{align*}
M(\Lambda)\leq \tfrac{1}{4} C_0^{-1}+\Lambda^{-1}\alpha_0,\quad \Lambda_0\leq 4\alpha_{0} C_0.
\end{align*}
Thus, the claim again follows from Proposition \ref{proposition stability under perturbations effective}.
\end{proof}

\section{Applications}\label{Section Applications}

\subsection{From Tomas--Stein to Kenig--Ruiz--Sogge}\label{subsection From Tomas--Stein to Kenig--Ruiz--Sogge}
In the remainder of this section we will consider applications of Theorem \ref{theorem spectral cluster to resolvent intro}. Ultimately, we will obtain new uniform resolvent estimates for compact manifolds with boundary or nonsmooth metrics. As a warm-up, we start by re-proving several known resolvent estimates, using the corresponding spectral cluster bound.

The prototype of a spectral cluster bound is the Tomas--Stein inequality, which can be stated as a spectral cluster bound for the Euclidean Laplacian $-\Delta$ (see \cite[Chapter 5]{MR3645429}),
\begin{align}\label{Tomas--Stein spectral cluster form}
\|\Pi_{[\lambda,\lambda+\epsilon]}\|_{L^2(\R^n)\to L^{q_n}(\R^n)}\lesssim \epsilon^{1/2}\lambda^{1/q_n}
\end{align}
for $\lambda,\epsilon>0$. This also follows from the unit size spectral cluster bounds
\begin{align}\label{Tomas--Stein spectral cluster form unit size}
\|\Pi_{[\lambda,\lambda+1]}\|_{L^2(\R^n)\to L^{q_n}(\R^n)}\lesssim \lambda^{1/q_n}
\end{align}
by scaling. We will show that \eqref{Tomas--Stein spectral cluster form unit size} implies the resolvent bound
\begin{align}\label{KRS}
\|(\Delta+1+\I 0)^{-1}\|_{L^{q_n'}(\R^n)\to L^{q_n}(\R^n)}\lesssim 
1.
\end{align}
This is the Kenig--Ruiz--Sogge bound in the special case of self-dual exponents (\cite[Theorem 2.3]{MR894584}) and with spectral parameter $|z|=1$. The general case can again be obtained by scaling. Moreover, all other bounds $2\leq q\leq 2^*$ follow from \eqref{KRS} by Sobolev embedding.
The Tomas--Stein estimate \eqref{Tomas--Stein spectral cluster form unit size} is of the form of the abstract spectral cluster bound \eqref{spectral cluster estimate strong theorem} with $\delta=\epsilon=1$ and $\gamma(q_n)=\sigma(q_n)=1/q_n$. Note that the exponent in \eqref{KRS} is $-2/(n+1)=2 /q_n-1$. Therefore, Theorem \ref{theorem spectral cluster to resolvent stronger version} implies
\begin{align}
\|(\Delta+(\lambda+\I)^2)^{-1}\|_{L^{q_n'}(\R^n)\to L^{q_n}(\R^n)}\lesssim \lambda^{-\frac{2}{n+1}}.
\end{align}
Rescaling $x\to\lambda^{-1} x$ then yields
\begin{align}
\|(\Delta+(1+\I\lambda^{-1})^2)^{-1}\|_{L^{q_n'}(\R^n)\to L^{q_n}(\R^n)}\lesssim 1.
\end{align}
Taking the limit $\lambda\to\infty$ proves \eqref{KRS}.

\subsection{Consequences of Sogge's bounds}
As explained in the introduction, combining Theorem \eqref{theorem spectral cluster to resolvent intro} with Sogge's bounds \eqref{Sogge}, we recover the result of Frank--Schimmer~ \cite{MR3620715} and Burq--Dos Santos Ferreira--Krupchyk~\cite{MR3848231}, and thus also that of Dos Santos Ferreira--Kenig--Salo \cite{MR3200351}. 

\subsection{Nonsmooth metrics}\label{Section nonsmooth}
As already discussed in Section 2, part (i) of Theorem \ref{theorem nonsmooth into} can be deduced from several references \cite{MR2262171,MR2094851,MR3900030,Smith resolvent}, either with or without the help of Theorem~\ref{theorem spectral cluster to resolvent intro}. The easiest route is to combine Smith's result \cite{MR2262171} with Theorem \ref{theorem spectral cluster to resolvent intro}. Alternatively, one can apply \cite[Theorem 2.5]{MR2094851} to the Weyl quantization of the symbol
$p(x,\xi)=\lambda^{-1}(\lambda^2-g_{\lambda}^{ij}(x)\xi_i\xi_j)$, localized to frequency $\lambda$. This is in $\lambda S_{\lambda,\lambda^{1/2}}$ after suitable global extension and satisfies (A2), (A3) in \cite{MR2094851} with $k=0$. 
This yields \eqref{ineq. first order equation (revision)} with $T=1$ and $\gamma(q)=\sigma(q)$. The remainder of the proof is identical to that of Theorem \ref{theorem spectral cluster to resolvent intro} (or rather Theorem \ref{theorem spectral cluster to resolvent stronger version}). The use of the Weyl quantization is immaterial since the error compared to $\Delta_g+\lambda^2$ can be absorbed into the forcing term $G$ in \eqref{ineq. first order equation (revision)}.

Part (ii) follows from \cite[Theorem 1.1]{MR3282983} and Theorem \ref{theorem spectral cluster to resolvent intro}.
At the endpoint $q=8$, the result could also be deduced (with the same logarithmic loss as for the spectral cluster bound in \cite{MR3282983}) without using Theorem \ref{theorem spectral cluster to resolvent intro} directly, but instead applying the Christ--Kiselev lemma to the last formula in Step 1 in the proof of \cite[Theorem 2.1]{MR3282983}.

Part (iii) of Theorem \ref{theorem nonsmooth into} is a consequence of Theorem~\ref{theorem spectral cluster to resolvent stronger version} (with $\eps=\lambda^{\rho}$, $\delta=\lambda^{\rho/2}$, $\gamma(q)=\sigma(q)$) and the following lemma.

\begin{lemma}\label{lemma large spectral windows Cs}
If $1\leq s\leq 2$ and $\rho=(2-s)/(2+s)$, then the spectral cluster estimates
\begin{align}\label{no loss estimate for long spectral windows}
\|\Pi_{[\lambda,\lambda+\lambda^{\rho}]}\|_{L^2\to L^q}\lesssim \lambda^{\sigma(q)+\rho/2}
\end{align}
hold for all $2\leq q\leq \infty$.
\end{lemma}

\begin{proof}
By interpolation with the $L^2\to L^{\infty}$ bound \cite[(5)]{MR2280790} we can reduce this to the bound for $q=q_n$. 
By the implication $(c)\implies (a)$ in Corollary \ref{corollary equivalence cluster and quasimode} it suffices to prove the quasimode bound
\begin{align}
    \|u\|_{L^{q_n}(M)}\lesssim \lambda^{1/q_n+\rho/2}(\|u\|_{L^2(M)}+\lambda^{-1}\|\nabla_gu\|_{L^2}+\lambda^{-1-\rho}\|(\Delta_g+\lambda^2)u\|_{L^2(M)}.
\end{align}
After some preliminary reductions as detailed in Section 4, the latter follows from summing the bounds in \cite[Corollary 7]{MR2280790} over cubes $Q_R$ (of sidelength $R=\lambda^{-\rho}$) contained in the unit cube.
\end{proof}

\begin{remark}
Lemma \ref{lemma large spectral windows Cs} shows that we have no loss estimates for spectral windows of length $\lambda^{\rho}$. Indeed, orthogonality and optimality of the universal estimates \eqref{Sogge}  show that \eqref{no loss estimate for long spectral windows} is best possible. Lemma \ref{lemma large spectral windows Cs} is in a sense a dual version of \cite[Theorem 2]{MR2280790}, which states no loss estimates for the spectral windows of unit length hold on the $\lambda^{-\rho}$ spatial scale.
\end{remark}

We conclude this subsection with resolvent estimates for $C^s$ metrics that are perfect with respect to the spectral region, but not with respect to the exponent. They rely on spectral cluster bounds of Smith \cite{MR2280790}. These are sharp at the critical exponent $q_n$, but improvements are expected for larger $q$. 

\begin{theorem}\label{Theorem Cs perfect region, imperfect q}
Let $\Delta_g$ be the Laplace--Beltrami operator on an $n$-dimensional compact boundaryless Riemannian manifold with $C^{s}$ metric, $1\leq s\leq 2$. If
\begin{align}\label{q in Cs case}
2\leq q\leq \tfrac{2n(s+2)+s-2}{(n-2)(s+2)}=:q_s,
\end{align}
and $q<\infty$ if $n=2$, then we have 
\begin{align}\label{Cs smaller q}
\|(\Delta_g+(\lambda+\I)^2)^{-1}\|_{L^{q'}\to L^{q}}\lesssim \langle\lambda\rangle^{2\gamma(q)-1},
\end{align}
where $\gamma(q)=\sigma(q)+\tfrac1q\tfrac{2-s}{2+s}\leq \tfrac12$. 
\end{theorem}

\begin{proof}
This follows from \cite[Corollary 3]{MR2280790} in combination with Theorem~\ref{theorem spectral cluster to resolvent intro}.
\end{proof}

\begin{remark}
Even though Theorem \ref{Theorem Cs perfect region, imperfect q} does not yield a uniform Sobolev inequality, i.e.\ $q<2^*$ if $n\geq 3$, the value of $q_s$ is surprisingly close to the Sobolev exponent. Figure 2 depicts the ratio $q_s/2^*$ for various values of $n$ and $s$.   
\begin{figure}[h]
\begin{center}
\begin{tabular}{|c|c|c|c|}
\hline
&$n=3$&$n=4$&$n=5$\\
\hline
$s=1$&$0.944444$&$0.958333$&$0.966667$\\
\hline
$s=1.3$&$0.964646$&$0.973485$&$0.978788$\\
\hline
$s=1.6$&$0.981481$&$0.986111$&$0.988889$\\
\hline
\end{tabular}
\caption{The ratio $q_s/2^*$ for various values of $n$ and $s$.}
\end{center}
\end{figure}

\end{remark}

\subsection{Manifolds with boundary}

\begin{proof}[Proof of Theorem \ref{theorem boundary intro}]
(i) follows from Theorem \ref{theorem nonsmooth into} (iii), see Remark \ref{remark proof of theorem boundary (i)}. 

(ii): Smith and Sogge proved the spectral cluster bound \eqref{abstract spectral cluster estimate into} with 
\begin{align*}
\gamma(q)=
\begin{cases}
\frac{4}{3}\sigma(q)&\quad \mbox{if}\quad \frac{2(n+2)}{n-1}\leq q\leq q_n,\\
\sigma(q)+\frac{1}{3}\big(\eps(q)-\frac{1}{q}\big)_-&\quad \mbox{if}\quad  q_n\leq q\leq  2^*,
\end{cases}
\end{align*}
where $x_-=\max(0,-x)$ and
$\eps(q)=(n-1)(1/2-1/q)-2/q\geq 0$. See \cite[Theorem 2.1]{MR2316270} for $n=2$ and \cite[Theorem 7.2]{MR2316270} for $n\geq 3$. In fact, they proved homogeneous square function estimates for solutions to the wave equation, which imply the spectral cluster estimates (see e.g.\ \cite{MR1308407}, \cite{MR2262171}). Since Theorem \ref{theorem spectral cluster to resolvent intro} applies only to boundaryless manifolds, we have referred to the corresponding results of \cite{MR2316270}. These results apply to boundaryless manifolds with special Lipschitz metrics (smooth, with a Lipschitz singularity across a hypersurface). Thus, we are in the setting of Theorem \ref{theorem spectral cluster to resolvent intro}, and the claim follows.

(iii): Smith and Sogge \cite{MR1308407} have also proved perfect spectral cluster estimates for the Dirichlet realization of the Laplace-Beltrami operator on a smooth manifold with strictly concave boundary, see Remark \ref{remark proof of theorem boundary (iii)}.
The bound in (iii) is an immediate consequence of these results and e.g.\ \eqref{abstract resolvent bound with log loss}, with $\delta,\epsilon=1$ there. 
\end{proof}

\subsection{Improved estimates}
To prove Theorem \ref{theorem improved intro} we only need to combine the spectral cluster bound \cite[Theorem 1]{Canzani--Galkowski} of Canzani and Galkowski with Theorem \ref{theorem spectral cluster to resolvent stronger version}. Note that this is the only instance where we need the version for improved bounds ($\epsilon<1$). The result in \cite[Theorem 1]{Canzani--Galkowski} is stated in a semiclassical framework, but it straightforward to translate it to a bound of the form \eqref{abstract quasimode bound}, or equivalently \eqref{abstract spectral cluster bound epsilon}. More precisely, \cite[Theorem 1]{Canzani--Galkowski}  asserts that 
\begin{align*}
\|u\|_{L^q}\lesssim h^{-\sigma(q)}\left(\frac{\|u\|_{L^2}}{\sqrt{\log h^{-1}}}+\frac{\sqrt{\log h^{-1}}}{h}\|(-h^2\Delta_g-1)u\|_{H_{\rm scl}^{\frac{n-3}{2}-\frac{n}{q}}}\right)
\end{align*}
for $0<h\ll 1$. Here we considered the special case $U=M$ in the theorem of Canzani--Galkowski and changed the notation to $p=q$, $\delta(p)=\sigma(q)$. We also translate this from the semiclassical to the high frequency setting by taking $h=1/\lambda$, obtaining
\begin{align*}
\|u\|_{L^q}\lesssim \frac{\lambda^{\sigma(q)}}{\sqrt{\log\lambda}}\|u\|_{L^2}+\lambda^{\sigma(q)-1}\sqrt{\log\lambda}\|(\Delta_g+\lambda^2)u\|_{L^2}.
\end{align*}
We used that the semiclassical Sobolev norms $H_{\rm scl}^2$ appearing in the previous estimate are innocuous in the frequency-localized regime, i.e.\ that
\begin{align*}
\|(-h^2\Delta_g-1)u\|_{H_{\rm scl}^{s}}\lesssim \|(-h^2\Delta_g-1)u\|_{L^2}
\end{align*}
for $u$ Fourier supported in $|\xi|\lesssim 1/h$. Now a direct comparison with \eqref{abstract quasimode bound}, with $\gamma(q)=\sigma(q)$, $\epsilon=1/\ln\langle\lambda\rangle$ and $\delta=\sqrt{\epsilon}$ yields the claimed resolvent estimate.

\subsection{Fractional Schrödinger operators}
We give a brief argument how the proof of Theorem \ref{theorem spectral cluster to resolvent intro} (or that of Theorem \ref{theorem spectral cluster to resolvent stronger version}) can be modified in such a way that Theorem \ref{theorem fractional intro} follows from the spectral cluster bounds of Smith \cite{MR2262171}. This also gives a flavor of other possible generalizations. The point of departure is the first order equation \eqref{first order equation (revision)}. This is still valid if $u$ satisfies $((-\Delta_g)^{\alpha/2}-(\lambda+\I)^{\alpha})u=\lambda^{\alpha-1}F$ (here we set $\delta,\epsilon=1$ and use a different normalization for the right hand side in the resolvent equation than in the case $\alpha=2$). The reason is that the symbol $p(x,\xi)=|\xi|_g^{\alpha}-\lambda^{\alpha}$ is of real principal type, with $|p_{\xi}|\gtrsim \lambda^{\alpha-1}$ on $\{p=0\}$ and at frequency $\lambda$. Since $g$ is one order smoother that in the assumptions of Theorem \ref{theorem spectral cluster to resolvent stronger version}, it would be sufficient to perform the symbol smoothing at frequency $|\xi|\ll\sqrt{\lambda}$ as opposed to $\lambda$. Then, in a conic region and in suitable coordinates,
\begin{align*}
p(x,\xi)=e(x,\xi)(\xi_1-a(x,\xi',\lambda)),
\end{align*}
where $a>0$ is a symbol in $\lambda C^2S_{\lambda,\sqrt{\lambda}}$, and $e(x,\xi)$ is elliptic and of size $\lambda^{\alpha-1}$. 
Let $Q$ be a parametrix for $e^w$, with a symbol in $\lambda^{1-\alpha}S_{\lambda,\lambda}$. As in the proof of Theorem \ref{theorem spectral cluster to resolvent stronger version} first order calculus yields
\begin{align}\label{reduction proof fractional}
(D_{x_1}-a_{\lambda}(x,D_{x'}))u=\lambda^{\alpha-1}QF+Ru
\end{align}
where $R=\mathcal{O}_{L^2\to L^2}(1)$ and the $x$-frequencies of $a_{\lambda}(x,\xi')$ are truncated to $|\xi|\ll\sqrt{\lambda}$.
We proceed as in the proof of Theorem \ref{theorem spectral cluster to resolvent stronger version}. Using the abstract results of Corollary \ref{corollary equivalence cluster and quasimode} (implication $(a)\implies (c)$), we first rewrite Smith's spectral cluster bound \cite{MR2262171} as a quasimode estimate
\begin{align}\label{quasimode proof fractional}
\|u\|_{L^{q}}\lesssim \lambda^{\sigma(q)}\|u\|_{L^{2}}+\lambda^{\sigma(q)-1}\|(\Delta_g+\lambda^2)u\|_{L^{2}},
\end{align}
then we reverse the steps in the reduction to \eqref{reduction proof fractional} to convert \eqref{quasimode proof fractional} to a microlocalized estimate 
\begin{align*}
\|u\|_{L^q}\lesssim \lambda^{\sigma(q)}(\|u\|_{L^2}+\|(D_{x_1}-a_{\lambda}^w(x,D_{x'}))u\|_{L^2}).
\end{align*}
This is exactly the same as \eqref{microlocalized quasimode} but with $\delta,\epsilon=1$. Thus the claimed resolvent estimate follows from Lemma \ref{lemma first order 2}, just as in the proof of Theorem \ref{theorem spectral cluster to resolvent stronger version}.

\subsection*{Acknowledgements} I would like to thank Rupert Frank for suggesting the topic of uniform resolvent estimates on manifolds with boundary, from which the present work emerged. I would also like to thank Herbert Koch, Hart Smith, Chris Sogge and Daniel Tataru for illuminating discussion during various stages of the project. I am grateful to Jeff Galkowski for discussions that lead to a version of Theorem \ref{theorem spectral cluster to resolvent stronger version} that is applicable to improved estimates. I would also like to thank the anonymous referees for their comments, which led to significant improvements in the presentation of Section 4.
Support by the Engineering \& Physical Sciences Research Council [grant
number EP/X011488/1], during the revision period of the article, is gratefully acknowledged.

\bibliographystyle{abbrv}

\end{document}